\documentclass[12pt]{article}

\usepackage[many,breakable]{tcolorbox}
\usepackage{varwidth}
\usepackage{graphicx}
\usepackage{amsmath,amsthm,amssymb,enumerate}
\usepackage{mathtools}
\usepackage{euscript,mathrsfs}
\usepackage{dsfont}
\usepackage{float}
\usepackage{xcolor}
\usepackage{empheq}
\usepackage[left=2cm,right=2cm,top=3.5cm,bottom=3.5cm]{geometry}
\mathtoolsset{showonlyrefs=false}
\usepackage{color}
\catcode`\@=11 \@addtoreset{equation}{section}

\catcode`\@=12

\tcbset{enhanced,breakable,arc=0mm,colframe=white,colback=black!10!white}

\newtheorem{Theorem}{Theorem}[section]
\newtheorem{Proposition}[Theorem]{Proposition}

\newtheorem{Lemma}[Theorem]{Lemma}
\newtheorem{Corollary}[Theorem]{Corollary}

\theoremstyle{definition}
\newtheorem{Definition}[Theorem]{Definition}

\newtheorem{Remark}[Theorem]{Remark}

\newcommand{\bTheorem}[1]{
	\begin{Theorem} \label{T#1} }
	\newcommand{\eT}{\end{Theorem}}

\newcommand{\bProposition}[1]{
	\begin{Proposition} \label{P#1}}
	\newcommand{\eP}{\end{Proposition}}

\newcommand{\bLemma}[1]{
	\begin{Lemma} \label{L#1} }
	\newcommand{\eL}{\end{Lemma}}
\newcommand{\bCorollary}[1]{
	\begin{Corollary} \label{C#1} }
	\newcommand{\eC}{\end{Corollary}}

\renewcommand{\(}{\left(}
\renewcommand{\)}{\right)}

\allowdisplaybreaks

\newcommand*{\bigchi}{\mbox{\large$\chi$}}
\renewcommand{\u}{\mathbf{u}}
\newcommand{\U}{\mathbf{U}}
\newcommand{\vv}{\mathbf{v}}

\newcommand{\dta}{\tau}
\newcommand{\ddta}{\,\mathrm{d}\tau}
\newcommand{\iQt}{\int_{\Omega_t}}

\newcommand{\CC}{\mathbf{C}}
\newcommand{\HH}{\mathbf{H}}
\newcommand{\DD}{\mathbf{D}}

\newcommand{\I}{\mathbf{I}}
\newcommand{\TT}{\mathbf{T}}
\newcommand{\tr}[1]{\mathrm{tr}\({#1}\)}
\newcommand{\trr}[1]{\mathrm{tr}({#1})}
\newcommand{\trC}{\mathrm{tr}(\CC)}
\newcommand{\trH}{\mathrm{tr}(\HH)}
\newcommand{\trCim}{\mathrm{tr}(\CC_m^{-1})}
\newcommand{\trCt}{\mathrm{tr}(\CC(t))}
\newcommand{\trCo}{\mathrm{tr}(\CC(0))}

\newcommand{\trlnC}{\mathrm{tr}(\log\CC)}
\newcommand{\trlnCt}{\mathrm{tr}(\log\CC(t))}
\newcommand{\trlnCo}{\mathrm{tr}(\log\CC(0))}

\newcommand{\Du}{\mathrm{D}\u}
\newcommand{\DU}{\mathrm{D}\U}
\renewcommand{\div}[1]{\mathrm{div}\left({#1}\right)}

\newcommand{\dib}[1]{\mathrm{div}\big({#1}\big)}
\newcommand{\p}{\partial}
\newcommand{\na}{\nabla}
\newcommand{\R}{\mathbb{R}}

\newcommand{\td}{\frac{\mathrm{d}}{\mathrm{d}t}}
\newcommand{\dx}{\,\mathrm{d}x}
\newcommand{\dd}{\,\mathrm{d}}
\newcommand{\dt}{\,\mathrm{d}t}

\DeclarePairedDelimiter{\inp}{\langle}{\rangle}
\DeclarePairedDelimiter{\norm}{\|}{\|}
\DeclarePairedDelimiter{\snorm}{|}{|}

\DeclareMathOperator{\spn}{span}

\newcommand\restr[2]{\ensuremath{\left.#1\right|_{#2}}}

\newcommand\bee{\begin{equation}}
\newcommand\eee{\end{equation}}
\newcommand\baa{\begin{equation}\begin{aligned}}
\newcommand\eaa{\end{aligned}\end{equation}}
\newcommand\e{\varepsilon}
\newcommand\de{\delta}
\newcommand\nn{\nonumber}

\def\softd{{\leavevmode\setbox1=\hbox{d}%
		\hbox to 1.05\wd1{d\kern-0.4ex{\char039}\hss}}}

\date{}

\begin{document}

	\title{Existence, regularity and weak-strong uniqueness for the three-dimensional Peterlin viscoelastic model}
	
	\author{Aaron Brunk \and Yong Lu \and M\' aria Luk\' a\v cov\' a-Medvi\softd ov\'a
	}
	
	\date{\today}
	
	\maketitle
	
	\bigskip
	
	\centerline{$^*$ Institute of Mathematics, Johannes Gutenberg-University Mainz}
	
	\centerline{Staudingerweg 9, 55128 Mainz, Germany}
	
	\centerline{abrunk@uni-mainz.de, \quad lukacova@uni-mainz.de}

	\bigskip
	
	\centerline{$^\dagger$ Department of Mathematics, Nanjing University}
	
	\centerline{ 22 Hankou Road, Goulou District, 210093 Nanjing, China}

	\centerline{luyong@nju.edu.cn}

	\maketitle
	
	\bigskip
	
	\begin{abstract}
		\noindent
		In this paper we analyze three-dimensional Peterlin viscoelastic model. By means of a mixed Galerkin and semigroup approach we prove the existence of weak solutions. Further, combining parabolic regularity with the relative energy method we derive a conditional weak-strong uniqueness result.
	\end{abstract}

	\tableofcontents
	
	\section{Introduction}
	Complex viscoelastic fluids find their applications in everyday life. They are used to model polymers, blood or even in food industry.
In this paper we study a model for complex viscoelastic fluids where the Peterlin approximation is applied in order to represent time evolution of the elastic conformation tensor, cf.~Peterlin \cite{Peterlin}. Dilute theory for complex fluids proposes that polymeric molecules can be represented by dumbbells that are suspended in a Newtonian solvent. A dumbbell is characterized by two beads connected by a spring. In general, a nonlinear spring force can be expressed as  $F(\mathbf{R})=\gamma(|\mathbf{R}|^2)\mathbf{R}, $ where $\mathbf{R}$ is the vector connecting the two beads. The Peterlin approximation replaces the length of a spring by an averaged length of springs leading to $F(\mathbf{R})\approx \gamma(\langle |\mathbf{R}|^2 \rangle)\mathbf{R}$, where $\langle \cdot \rangle$ denotes the average over all configurations.
	We note in passing that the well-known Oldroyd-B model is based on the linear Hookean law \cite{Barrett.2018b}, i.e. $F(\mathbf{R})=  H \mathbf{R}$, where $H$ is the spring constant. For detailed modeling we refer to the monographs \cite{Bird1,Bird2,Larsonbook,joseph,renardy2000mathematical,Siginer2014}.  The  free energy of viscoelastic fluid models was discussed by Lelièvre et al. in \cite{cms/1199377557}.

	In typical models for viscoelastic fluids, such as the Oldroyd-B \cite{Barrett.2011, Chupin.2018} or FENE-P \cite{Barrett.2018} model, the diffusive elastic terms are usually omitted yielding a hyperbolic equation for the time evolution of the elastic stress tensor \cite{renardy2000mathematical}. However, if the center-of-mass diffusion of polymer dumbbells is taken into account, a diffusive term arises in the evolution equation of the elastic stress tensor leading to a parabolic model, see e.g. Barrett and S\"uli \cite{Barrett.2011b} and the references therein.
	
	The aim of this paper is to extend and generalize our previous analytical results \cite{LukacovaMedvidova.2015,LukacovaMedvidova.2017} for the two-dimensional Peterlin viscoelastic model to three space dimensions. Besides showing global in time existence of weak solutions we also prove the weak-strong uniqueness result by means of the relative energy method.
	The Peterlin viscoelastic model reads

	\begin{tcolorbox}
		\begin{align}	
		\frac{\partial\u   }{\partial t} + (\u\cdot\nabla)\u  &= \dib{\eta\Du} -\nabla p + \mbox{div}\,\TT, \nonumber \\
		\frac{\partial\CC  }{\partial t} + (\u\cdot\nabla)\CC &= (\nabla\u)\CC + \CC(\nabla\u)^\top + \Phi(\trC)\mathbf{I} -\bigchi(\trC)\CC + \varepsilon\Delta\CC \label{eq:uc_model}, \\
		\mbox{div}\,{\u}&=0,\quad \TT = \trC\CC. \nonumber
		\end{align}
	\end{tcolorbox}
	 Here $\Phi:= \trC + a$ and $\bigchi:=\trC^2 + a\snorm*{\trC}$ for a given $a\geq 0$. System (\ref{eq:uc_model}) is considered on $\Omega\times(0,T)$, where $\Omega\subset\mathbb{R}^3$ is a $C^{2,\beta}$ domain, $\beta\in (0,1)$. It is equipped with the following initial and boundary conditions
	\begin{align}
	\restr{(\u,\CC)}{t=0} = ( \u_0, \CC_0),  \quad  \restr{\u}{\p\Omega} = \mathbf{0}, \quad \restr{\p_n\CC}{\p\Omega} = \mathbf{0}. \label{eq:bc}
	\end{align}
The Peterlin model \eqref{eq:uc_model} consists of the incompressible Navier-Stokes equations for the evolution of velocity $\u$, that are coupled in a nonlinear way to a time evolution of the conformation tensor $\CC$. The conformation tensor $\CC$
is related to the mean deformation of the polymer molecules and
 models the elastic effects of a viscoelastic fluid. Further, the functions $\Phi,\bigchi$ represent generalized relaxation terms. The diffusion term $\varepsilon\Delta\CC$ models the center-of-mass diffusion of the polymer chains. The remaining terms of the evolution equation for the conformation tensor represent the upper convected derivative. This is a frame invariant time derivative for matrices, see \cite{joseph} for more details. The coupling with the Navier-Stokes equations is realized via the elastic stress tensor denoted by $\TT$.

 Mathematical literature dealing with analysis of viscoelastic fluid models is very broad. Existence and uniqueness of strong $L^p$ solutions for large times and small data or local well-posedness for large data was studied by Geissert et al. in \cite{Geissert2012} for generalized Oldroyd-B or Peterlin models. Summary of results of local existence of strong solutions can be found, e.g., in \cite{FERNANDEZCARA2002,Guillope.1990} and the references therein.  
Lei et al.~\cite{lei2007b,lei2007a} and Lin et al.~\cite{Lin2005}  obtained the  local and global existence results for smooth solutions of  the Oldroyd-B type models based on time evolution of the deformation tensor. In the context of the deformation tensor the global existence in critical spaces was studied by Zhand and Fang in \cite{Zhang2012}. Applying the Lagrangian formulation the global well-posedness result was obtained by He and Xu in \cite{He2010}. In  \cite{Chupin.2018} Chupin proved the existence of global strong solutions for some viscoelastic models  deriving suitable bounds on the stress tensor.
Global existence of weak solutions is a more delicate problem. Lions and Masmoudi proved in \cite{LIONS2000} global existence of two- and three-dimensional co-rotational Oldroyd-B model. However, in the evolution equation for the elastic stress tensor $\na\u$ is replaced by $\tfrac{1}{2}(\na\u-\na\u^\top)$.
Chemin and Mamsoudi analyzed in \cite{Chemin2001} local and global well-posedness in critical Besov spaces. Non-blow up criteria for the Oldroyd-B models were presented by Lei et al.~\cite{Lei2010} and Chen and Zhang \cite{Chen2006}.
Diffusive  Oldroyd-B model has been studied by Constantin and Kliegl \cite{Constantin.2012} and Barrett and Boyaval in \cite{Barrett.2011}. Global existence of weak solutions has been proven for two-dimensional viscoelastic fluids.  Global existence of generalized, the so-called dissipative solutions in three space dimensions, was investigated in \cite{Kalousek2019} and weak-strong uniqueness result was shown.

The FENE-P model, that is based on the Peterlin approximation and a finitely extensible nonlinear elastic (FENE) spring potential has been studied by Masmoudi in \cite{Masmoudi.2011} and global existence of weak solutions has been shown in three space dimensions.
Existence of weak solutions for diffusive macro-micro model based on the FENE or Hookean spring was studied by Barrett and Süli in \cite{Barrett.2011b,Barrett.2012}, see also \cite{Barrett.2016b,Barrett.2016,Lu.2018} for further developments for compressible viscoelastic fluids.

We conclude this introductory part by referring to a recent work of Bathory, Buli\v{c}ek and M\'alek \cite{Bathory.2020}, where global in time existence of weak solution {was shown} for a special rate-type fluids in three space dimensions.

The diffusive Peterlin model has been studied from analytical and numerical point of view in our recent works \cite{Gwiazda2018,LukacovaMedvidova.2015,LukacovaMedvidova.2017a,LukovMedvidov2017b}.
However, due to missing a priori estimates for the conformation tensor, all of these results are restricted to two-dimensional viscoelastic flows. In the present paper we combine techniques from \cite{Barrett.2018} and \cite{LukacovaMedvidova.2017} in order to extend the global existence result to three space dimensions.

	Moreover, we analyze the properties and regularity of the conformation tensor. It turns out that the physically relevant positive-(semi) definiteness, see \cite{Hulsen.1990}, is crucial to obtain the existence result in three space {dimensions}. Furthermore, we prove a weak-strong uniqueness result which is rarely seen in this context except from \cite{Barrett.2017}. In contrast to the Oldroyd-B model which is studied in \cite{Barrett.2017}  our model is more complex and  needs a special treatment. \\
	
	The structure of this paper is as follows. In Section 2 we introduce suitable {notations} and recall some necessary analytical tools. In Section 3 we state the concept of weak {solutions} which we use and formulate the main results. Section 4 deals with the existence proof and the energy inequality, while Section 5 focuses on the parabolic regularity and the energy equality. Finally, in Section 6 we apply the previous results in a relative energy method to obtain the conditional weak-strong uniqueness result. Section 7 illustrates a possible application of the relative energy method in the context of convergence of numerical schemes.

	\section{Preliminaries}
	In this section we introduce {suitable notations} and the theoretical framework for the upcoming analysis of the Peterlin viscoelastic model (\ref{eq:uc_model}). Let $d\in\{2,3\}$ denote the space dimension. We denote by $\Omega_T:=\Omega\times(0,T)$ and $\Omega_t:=\Omega\times(0,t)$ the full space-time cylinder and the intermediate space-time cylinder, respectively. The norm of the Lebesgue space $L^p(\Omega)$ is denoted by $\norm*{\cdot}_p$ and the norm of the Bochner space $L^p(0,T;L^q(\Omega))$ by $\norm*{\cdot}_{L^p(L^q)}$. Further, we set
	\begin{equation*}
	L^2_{\text{div}}(\Omega):= \overline{C_{0,\text{div}}^\infty(\Omega)}^{\norm*{\cdot}_2} \text{ and }  L^2(\Omega)^{d\times d}_{SPD}:=\{\DD\in L^2(\Omega)^{d\times d} \mid \vv^T\DD\vv > 0, \; \forall \vv\neq 0\in L^\infty(\Omega)^d\}.
	\end{equation*}
	 Here $C_{0,\text{div}}^\infty(\Omega)$ stands for the set of smooth divergence free functions that are compactly supported in $\Omega$.  We use the standard {notations} for the Sobolev spaces and set
	\begin{equation*}
	V:=H^1_{0,\text{div}}(\Omega)^d, \hspace{3.5em} H:=L^2_{\text{div}}(\Omega)^d.
	\end{equation*}
	Dual spaces of $ H^1_{0}(\Omega), W_{0}^{1,p}(\Omega), V$ are denoted by $H^{-1}(\Omega), W^{-1,p^*}(\Omega), V^*$,  respectively. The deformation gradient is the symmetric part of the velocity gradient given by
	\begin{equation*}
	    \Du = \frac{1}{2}(\na\u+\na\u^\top).
	\end{equation*}

	\begin{Proposition}[$L^p$-Matrix norm \cite{Mizerova.2015}]
		\label{prop:matineq}
		For a matrix valued function $\DD\in\R^{d\times d}$ and $p \geq 2$ we have
		\begin{equation*}
		\norm*{\mathrm{tr}(\DD)}^p_{p} := \int_\Omega \left(\sum_{i=1}^d \DD_{ii}\right)^p \dx \leq d^{p-1}\int_\Omega \sum_{i,j=1}^d \snorm*{\DD_{ij}}^p \dx=: d^{p-1}\norm*{\DD}^p_p.
		\end{equation*}
		For symmetric positive(-semi) definite matrices both norms are equivalent, i.e.
		\begin{equation*}
		\norm*{\DD}^p_p \leq \norm*{\mathrm{tr}(\DD)}^p_{p}  \leq d^{p-1}\norm*{\DD}^p_p.
		\end{equation*}
	\end{Proposition}
	The norm $\norm*{\mathrm{tr}(\DD)}_p$ is the so-called trace norm. We denote by $\CC:\DD =  \sum_{i,j=1}^d \CC_{ij}\DD_{ij}$ the Frobenius inner product of two matrices. For higher order tensors, such as gradients of matrices we use the inner product $\na\CC::\na\DD=\sum_{k=1}^{d} \partial_{x_{k}}\CC:\partial_{x_{k}} \DD$.
	
	\begin{Definition}
		Let $\DD(t)\in\R^{d\times d}$  be a symmetric matrix function that is diagonalized as follows
		 \begin{equation*}
		\DD(t) = \mathbf{Q}(t)\mathbf{\Lambda}(t)\mathbf{Q}^\top(t)
		\end{equation*}
		for all $t\in [0,T)$. Here $\Lambda(t)$ denotes the diagonal matrix containing the eigenvalues of $\DD(t)$, while $\mathbf{Q}(t)$ denotes the transformation matrix containing the eigenvectors of $\DD(t)$. We define the matrix logarithm $\log(\DD(t))$ for a symmetric positive definite (SPD) matrix function $\DD(t)$ by
		\begin{equation*}
		\log\DD(t) = \mathbf{Q}(t)\log(\mathbf{\Lambda})(t)\mathbf{Q}^\top(t).
		\end{equation*}	
		Furthermore, for $\DD\in C^1([0,T))$ the following Jacobi formula holds
		\begin{equation}
		\frac{\dd \DD}{\dt}:\DD^{-1} = \tr{\DD^{-1}\frac{\dd \DD}{\dt}}=\frac{\dd}{\dt}\trr{\log\DD}. \label{eq:jacobi}
		\end{equation}
	\end{Definition}

	\begin{Lemma}[\cite{Mizerova.2015}]\label{lem_logid}
	   Let $\DD$ be a symmetric positive definite matrix. Then the following holds:
	    \begin{align}
	        \trr{\log\DD} = \log \det(\DD),\quad \trr{\DD}^2 - 2\trr{\log\DD} - \trr{\I} \geq 0,\quad \trr{\DD+\DD^{-1}-\I} \geq 0.\label{eq:logid}
	    \end{align}
	\end{Lemma}
	
	\begin{Lemma}[\cite{Barrett.2018}]
		\label{lem:matinv}
		Let  $\DD\in H^2(\Omega)^{m\times m}\cap C^1(\overline{\Omega})^{m\times m},m\in\mathbb{N}$, be a symmetric matrix function, which is uniformly positive definite on $\overline{\Omega}$ and satisfies homogeneous Neumann boundary conditions, then
		\begin{equation}
		\label{eq:invconf}
		\int_\Omega \Delta\DD:\DD^{-1} \dx = - \int_\Omega \nabla\DD::\nabla\DD^{-1} \dx \geq \frac{1}{m}\int_\Omega \snorm*{\nabla\mathrm{tr}(\log\DD)}^2 \dx.
		\end{equation}
	\end{Lemma}
	
	\begin{Lemma}[Gronwall]
		\label{lem:gronwall}
		Let $f\in L^1(t_0,T)$ be non-negative and $g,\phi$ continuous functions on $[t_0,T]$. If $\phi$ satisfies
		\begin{align*}
		\phi(t) \leq g(t) + \int_{t_0}^t f(s)\phi(s) \mathrm{d}s, \text{ for all } t\in[t_0,T]
		\end{align*}
		then
		\begin{equation*}
		\phi(t) \leq g(t) + \int_{t_0}^t f(s)g(s)\exp{\(\int_s^t f(\tau) \mathrm{d}\tau\)} \mathrm{d}s, \text{ for all } t\in[t_0,T].
		\end{equation*}
		If $g$ is moreover non-decreasing, then
		\begin{equation*}
		\phi(t) \leq g(t)\exp{\(\int_{t_0}^t f(\tau) \mathrm{d}\tau\)}, \text{ for all } t\in[t_0,T].
		\end{equation*}
	\end{Lemma}

	\medskip

	We proceed by recalling some regularity results for parabolic Neumann problems. We start by introducing fractional-order Sobolev spaces. Let $\Omega$ be the whole space $\R^d$ or a bounded Lipschitz domain in $\R^d$. For any $k\in \mathbb{N}$, $\beta \in (0,1)$ and $s\in [1,\infty) $, we define
	\[
	W^{k+\beta,s}(\Omega):=\big\{v\in W^{k,s}(\Omega) : \| v \|_{W^{k+\beta,s}(\Omega)}<\infty \big\},
	\]
	where
	$$
	\| v \|_{W^{k+\beta,s}(\Omega)}:=\| v \|_{W^{k,s}(\Omega)} + \sum_{|\alpha|=k} \left(\int_\Omega\int_\Omega
	\frac{|\p^\alpha v(x)- \p^\alpha v(y)|^s}{|x-y|^{d+\beta s}} \dx \,{\rm d}y\right)^{\frac{1}{s}}.
	$$
	Consider the parabolic initial-boundary value problem:
	\baa\label{para-ib-Neumann}
	&\p_t \rho - \e\, \Delta \rho = h \ \mbox{in} \ \Omega_T;
	\quad \rho(0,\cdot) = \rho_0 \ \mbox{in} \   \Omega; \quad \p_{\bf n}\rho =0 \ \mbox{in} \ (0,T)\times \p \Omega.
	\eaa
	Here $\e>0$, $\rho_0$ and $h$ are known functions, and $\rho$ is the unknown solution. In what follows the regularity result will be useful, see for example Section 7.6.1 in \cite{Novotny.2004}.
	
	\begin{Lemma}\label{lem-parabolic-1}
		Let $0<\beta < 1, \ 1<p,q<\infty$,  $\Omega\subset \R^d$ be a bounded $C^{2,\beta}$ domain with $\beta\in (0,1)$,
		$$
		\rho_0 \in W^{2-\frac{2}{p},q}_{\bf n},\quad h\in L^p(0,T;L^q(\Omega)),
		$$
		where $W^{2-\frac{2}{p},q}_{\bf n}$ is the completion of the linear space $\{v \in  C^\infty(\overline \Omega): \ \p_{\bf n} v|_{\p \Omega} = 0 \}$ with respect to the norm of $W^{2-\frac{2}{p},q}(\Omega)$.
		Then there exists a unique function $\rho$ satisfying
		$$
		\rho \in L^p(0,T;W^{2,q}(\Omega)) \cap C([0,T];W^{2-\frac{2}{p},q}(\Omega)), \quad \p_t \rho \in L^p(0,T;L^q(\Omega))
		$$
		solving \eqref{para-ib-Neumann} in $\Omega_T$. In addition, $\rho$ satisfies the Neumann boundary condition in $\eqref{para-ib-Neumann}$ in the sense of the normal trace, which is well defined since $\Delta \rho \in L^p(0,T;L^{q}(\Omega))$.
		Moreover,
		\begin{align*}
		\e^{1-\frac{1}{p}}\|\rho\|_{L^\infty(W^{2-\frac{2}{p},q})}
		+ \left\|\partial_t\rho \right\|_{L^p(L^q)}
		+ \e \|\rho\|_{L^p(W^{2,q})}  \nonumber\\
		\leq C(p,q,\Omega)\big[\e^{1-\frac{1}{p}}
		\|\rho_0\|_{W^{2-\frac{2}{p},q}(\Omega)} + \|h\|_{L^p(L^q)}\big].
		\end{align*}
		\end{Lemma}
		
		Next we introduce the space of weakly continuous functions over a Banach space $X$ as follows
		$$ C_w([0,T];X):=\{ f:[0,T]\to X: \int_\Omega fv\in C([0,T]) \text{ for any } v\in X\}.
		$$ Here the integral is understood as a dual paring between $f$ and $v$, i.e. $f\in X^*$.
		
		\begin{Lemma}[\cite{Barrett.2016}]
		Let $X,Y$ be Banach spaces and assume $X$ is reflexive and is continuously embedded in $Y$, then
			\begin{equation*}
			L^\infty(0,T;X)\cap C_{w}([0,T];Y)=C_{w}([0,T];X).
			\end{equation*}
			
		\end{Lemma}

	\section{Main results}
	The aim of this section is to present the main results of the paper: the global existence of weak solutions in three space dimensions and the weak-strong uniqueness result. We start by defining the weak solutions to the Peterlin viscoelastic system \eqref{eq:uc_model}.
	
	\begin{tcolorbox}[breakable]
	
			\begin{Definition}
				Given the initial data $\left(\u_0,\CC_0 \right) \in [H \times L^2(\Omega)^{3\times3}_{SPD}]$.
				\label{defn:weak_sol_UC}
				The couple $(\u,\CC)$ is called a weak solution of (\ref{eq:uc_model}) in $\Omega_T$ if
				\begin{align*}
				&\u\in C_{w}([0,T];L^2(\Omega))\cap L^2(0,T;V) \cap C([0,T];L^{q}(\Omega))\cap W^{1,\frac{4}{3}}(0,T;V^*),\\
				&\CC\in  C_{w}([0,T];L^2(\Omega)) \cap L^2(0,T;H^1(\Omega))\cap L^4(\Omega_T) \cap C([0,T];L^{q}(\Omega))\cap W^{1,\frac{4}{3}}(0,T;H^{-1}(\Omega)),\\
				&\bigchi(\trr{\CC})\CC \in L^{\frac{4}{3}}(\Omega_T),\qquad \Phi(\trr{\CC}) \in L^2(\Omega_T),
				\end{align*}
				for any $1\leq q <2$ and  \eqref{eq:uc_model} is satisfied in the sense of distributions:
				\begin{itemize}
					\item For any $t \in (0,T]$ and  any test function $\vv \in C^{\infty}([0,T]; C_{c}^{\infty}(\Omega;\R^{3})),$
					\baa\label{eq:weak_sol_UC-1}
					& \int_{0}^{t} \int_{\Omega} \u \cdot  \frac{\p \vv}{\p t} \dx \ddta - \int_{0}^{t} \int_{\Omega} (\u\cdot\nabla)\u\cdot\vv \dx \ddta - \int_{0}^{t} \int_\Omega \eta\Du:\mathrm{D}\vv \dx \ddta  \\
					& =  \int_{0}^{t} \int_\Omega\trC\CC:\nabla\vv \dx \ddta   +  \int_{\Omega} \u (t) \cdot  \vv (t)\dx -   \int_{\Omega} \u_{0}  \cdot  \vv (0)\dx.
					\eaa
				\end{itemize}
				\begin{itemize}
					\item For any $t \in (0,T]$ and any test function $\DD \in C^{\infty}([0,T]\times \overline\Omega;\R^{3\times 3})),$
					\baa\label{eq:weak_sol_UC-2}
					& \int_{0}^{t} \int_{\Omega} \CC :  \frac{\p \DD}{\p t} \dx \ddta - \int_{0}^{t} \int_{\Omega} (\u\cdot\nabla)\CC : \DD \dx \ddta - \int_{0}^{t} \int_\Omega \e \nabla \CC :: \nabla \DD \dx \ddta  \\
					& =  \int_{0}^{t}\int_\Omega \bigchi(\trC)\CC:\DD \dx\ddta - \int_{0}^{t} \int_\Omega \Phi(\trC)\, {\rm tr}\,\DD \dx \ddta  \\
					& \quad +  \int_{\Omega} \CC (t) :  \DD (t)\dx -   \int_{\Omega} \CC_{0}  : \DD (0)\dx.
					\eaa
				\end{itemize}	
			\end{Definition}
	\end{tcolorbox}
We are now in the place to state our main result on existence of {\sl weak dissipative solutions}, i.e. weak solutions for which a suitable energy functional is decreasing in time.
	
	\begin{tcolorbox}
		\begin{Theorem}[Existence of weak dissipative solutions]\label{thm:existence}
			For given initial data $(\u_0,\CC_0)\in [H\times L^2(\Omega)_{SPD}^{3\times 3}]$ and any $T>0$ there exists a global in time weak solution of the Peterlin viscoelastic system (\ref{eq:uc_model}) in the sense of Definition \ref{defn:weak_sol_UC}. Moreover, it satisfies for a.e. $t\in(0,T)$ the energy inequality
			\begin{align}
			&\(\int_\Omega \frac{1}{2}\snorm*{\u(t)}^2 + \frac{1}{4}\snorm*{\trCt}^2 \dx \) +\int_{\Omega_t}\eta\snorm*{\Du}^2 + \frac{\varepsilon}{2}\snorm*{\na\trC}^2 + \frac{1}{2} \snorm*{\trC}^4 + \frac{a}{2}\snorm*{\trC}^3 \dx\ddta \nonumber \\
			&\leq  \frac{1}{2}\int_{\Omega_t} \snorm*{\trC}^2 + \frac{a}{2}\trC \dx\ddta + \(\int_\Omega \frac{1}{2}\snorm*{\u(0)}^2 + \frac{1}{4}\snorm*{\trCo}^2 \dx \). \label{eq:uc_energy}
			\end{align}
			If $a=0$ the conformation tensor $\CC$ is symmetric positive semi-definite a.e. in $\Omega\times[0,T)$.\\
			If $a>0$ and the initial datum $\CC_0$ satisfies additionally $\trr{\log\CC_0}\in L^1(\Omega)$ then $\CC$ is symmetric positive definite a.e. in $\Omega\times[0,T)$ and enjoys the additional regularity
			\begin{align*}
			\trr{\log\CC} &\in L^\infty(0,T;L^1(\Omega))\cap L^2(0,T;H^1(\Omega)), \\
			\trr{\CC^{-1}}, \trr{\CC^{-1}}\trC &\in L^1(0,T;L^1(\Omega)).
			\end{align*}
			Furthermore, if $a>0$, for a.a. $t\in(0,T)$ the free energy inequality holds, i.e.
			\begin{align}
			&\(\int_\Omega \frac{1}{2}\snorm*{\u(t)}^2 + \frac{1}{4}\snorm*{\trCt}^2 - \frac{1}{2}\trlnCt \dx \) \label{eq:uc_freeenergy} \\
			&+ \iQt \eta\snorm*{\Du}^2+ \frac{\varepsilon}{2}\snorm*{\na\trC}^2 + \frac{\varepsilon}{6}\snorm*{\na\tr{\log\CC}}^2 + \frac{1}{2}\bigchi(\trC)\mathrm{tr}(\TT + \TT^{-1} - 2\I )\dx\ddta  \nonumber\\
			\leq&  \(\int_\Omega \frac{1}{2}\snorm*{\u(0)}^2 + \frac{1}{4}\snorm*{\trCo}^2 - \frac{1}{2}\trlnCo \dx \).  \nonumber
			\end{align}
		\end{Theorem}	
	\end{tcolorbox}
	
	\begin{Remark}
The energy inequality \eqref{eq:uc_energy}  holds for $a \geq 0.$
The free energy inequality \eqref{eq:uc_freeenergy} is only valid in the case $a>0$, since this requires
positive definiteness of the conformation tensor $\CC$. The free energy in \eqref{eq:uc_freeenergy}
is positively bounded from below due to \eqref{eq:logid}.
	\end{Remark}

	\begin{Remark}
		Note that in \eqref{eq:uc_freeenergy} one can also write the energy inequality with $\frac{\varepsilon}{2}\snorm*{\CC^{-1/2}\na\CC\CC^{-1/2}}^2$ instead of $\frac{\varepsilon}{6}\snorm*{\na\tr{\log\CC}}^2$, see \cite{Malek.2018}. In this case  an energy equality holds for a smooth solution.
	\end{Remark}

	\begin{Remark}
		The positive definiteness condition $a>0$ can also be found for more general models, see \cite{Hulsen.1990}. The proof of the above result will be given in the next section.
		
		First by introducing a Galerkin approximation only for the velocity $\u$ we obtain by parabolic regularity a solution $(\u_m,\CC_m)$, where $\CC_m$ is the classical solution of $(\ref{eq:uc_model})_2$ corresponding to  finite-dimensional velocity $\u_m$. By the energy method and parabolic regularity we {can} obtain approximation independent bounds and we can further pass to the limit in the equations for $\u_m$ and $\CC_m$ and in the energy inequality. Furthermore, if $a>0$ we prove the positive definiteness of $\CC_m$ and obtain the desired free energy inequality (\ref{eq:uc_freeenergy}).
	\end{Remark}

	Our next result is devoted to the regularity of the conformation tensor, a piece of information that was missing in \cite{LukacovaMedvidova.2015,LukacovaMedvidova.2017}.
	
	\begin{tcolorbox}
		
		\begin{Theorem}[Conditional energy equality]\label{theo:cond-energy-eq}
			Moreover, if the initial datum $\CC_{0}\in W^{\frac{1}{2}, \frac{4}{3}}_{{\bf n}}(\Omega)$, then $\CC$ satisfies the following higher-order estimates
			\baa\label{2nd-time-est-C}
			\CC\in L^{\frac{4}{3}}(0,T;W^{2,\frac{4}{3}}(\Omega)) + L^{\tilde s}(0,T;W^{2,\tilde r}(\Omega)), \quad \frac{\p \CC}{\p t}\in  L^{\frac{4}{3}}(\Omega_T) + L^{\tilde s}(0,T;L^{\tilde r}(\Omega)),
			\eaa
			where $(\tilde s, \tilde r)$ satisfies
			\baa\label{s-r-tilde}
			\frac{2}{\tilde s} + \frac{3}{\tilde r} = 4, \quad 1<\tilde s <2, \ 1<\tilde r < \frac{3}{2}.
			\eaa

			If the weak solution $(\u, \CC)$ obtained in Theorem \ref{thm:existence} satisfies $\CC \in L^{s}(0,T; L^{r}(\Omega))$ with
			\baa\label{Serrin-s-r}
			\frac{2}{s} + \frac{3}{r} \leq 1, \quad 2<s<\infty,  \ 3 < r <\infty,
			\eaa
			then for almost all $\tau \in (0,T]$ there holds
			\baa\label{cond-energy-eq}
			&\int_{\Omega} \frac{1}{2}|\CC (\tau)|^{2}\dx + \int_{0}^{\tau}\int_{\Omega}  \varepsilon |\nabla \CC|^{2}\dx \dt + \int_{0}^{\tau}\int_{\Omega}  |\trC|^{2} |\CC|^{2} + a |\trC| |\CC|^{2}\dx \dt  \\
			& = \int_{\Omega} \frac{1}{2}|\CC_{0}|^{2}\dx + \int_{0}^{\tau}\int_{\Omega}  |\trC|^{2} + a \trC \dx\dt  + \int_{0}^{\tau}\int_{\Omega}  \left[ (\nabla\u)\CC + \CC(\nabla\u)^\top\right] : \CC \dx\dt.
			\eaa
			
		\end{Theorem}
	\end{tcolorbox}
	
	We give some remarks on the additional integrability assumption.
	\begin{Remark}\phantom{2em}
		\begin{itemize}
			\item  Following  the well known Serrin type blow-up criterion on the Leray-Hopf weak solutions to the three dimensional incompressible Navier-Stokes equations, see \cite{Serrin.1962, Struwe.1988, Berselli.2002, Kim.2006}, it can be shown that if the weak solution $(\u, \CC)$ satisfies $\u, \CC \in L^{s}(0,T;L^{r}(\Omega))$ with $(s,r)$ satisfying \eqref{Serrin-s-r}, then the weak solution $(\u, \CC)$ is regular.

			\item We only assume $\CC \in L^{s}(0,T;L^{r}(\Omega))$, while the Serrin type criterion requires both $\u$ and $\CC$ are in $L^{s}(0,T;L^{r}(\Omega))$ with $s, r$ satisfying \eqref{Serrin-s-r}.
			\item A similar result can be obtained by assuming $\u\in L^4(\Omega_T)$.
			\item In two space dimensions the first result holds with $\tilde s=\tilde r=\frac{4}{3}$. Furthermore, the second result holds without further integrability assumptions.
		\end{itemize}
		
	\end{Remark}
	
	\medskip

	We continue by introducing the relative energy $\mathcal{E}$ as the first-order Taylor expansion of the energy. Let $\u,\U$ and $\CC,\HH$ be two velocity vectors and conformation tensors, respectively. We introduce the relative energy
	\begin{equation}
	 \mathcal{E}(\u,\CC\vert\U,\HH) := \mathcal{E}_{kin} + \mathcal{E}_{el} + \mathcal{E}_{frob} \label{eq:rel_en_full}
	\end{equation}
	where
	\begin{align*}
	\mathcal{E}_{kin}(\u\vert\U)=\frac{1}{2}\int_\Omega \snorm*{\u-\U}^2 \dx, \quad \mathcal{E}_{el}(\CC\vert\HH)=\frac{1}{4}\int_\Omega \snorm*{\trr{\CC-\HH}}^2 \dx, \quad \mathcal{E}_{frob}(\CC\vert\HH)=\frac{1}{2}\int_\Omega \snorm*{\CC-\HH}^2\dx.
	\end{align*}
	Since the elastic relative energy is not definite, i.e. $\mathcal{E}_{el}(\CC\vert\HH)=0 \nRightarrow \CC=\HH$, we penalize the relative energy by $\norm*{\CC-\HH}_2^2$, i.e. the relative Frobenius energy.
	
	Due to the norm properties the following properties of $\mathcal{E}$ hold
	\begin{align}
	1)\;& \mathcal{E}(\u,\CC\vert\U,\HH)(t) = 0 \Longleftrightarrow \u(t)=\U(t), \CC(t)=\HH(t) \text{ a.e. in } \Omega, \label{eq:relen_prop}\\
	2)\;& \mathcal{E}(\u,\CC\vert\U,\HH)\geq 0 \text{ and }
	\mathcal{E}(\u,\CC\vert\U,\HH) \geq \begin{cases*}
	\frac{1}{2}\norm*{\u-\U}^2_2, \\
	\frac{1}{4}\norm*{\trr{\CC-\HH}}^2_2, \\
	\frac{1}{2}\norm*{\CC-\HH}^2_2.
	\end{cases*} \nonumber
	\end{align}
	We note, that the relative energies are often not induced by norms. Nevertheless, the above properties would still hold and follow from convexity of the corresponding energy, see for instance \cite{Feireisl2012}. For further study it is convenient to set
	\begin{align}
	\mathcal{E}_1(\u,\CC\vert\U,\HH):=\mathcal{E}_{kin} + \mathcal{E}_{el}, \quad \mathcal{E}_2(\CC\vert\HH):=\mathcal{E}_{frob}. \label{eq:relen12}
	\end{align}
    We note by passing that the relative energy resulting from the energy inequality \eqref{eq:uc_energy} is more convenient for the forthcoming investigations than those derived from the free energy, cf. \eqref{eq:uc_freeenergy}.

	\begin{tcolorbox}
		\begin{Proposition}[Relative energy inequality]
			\label{theo:relative_energy}
			Let $(\u,\CC)$ be a global weak solution of the Peterlin viscoelastic system (\ref{eq:uc_model}) starting from the initial data $(\u_0,\CC_0)$. Assume that $\CC \in L^{s}(0,T; L^{r}(\Omega))$ with $(s,r)$ satisfying \eqref{Serrin-s-r}.  Let $(\U ,\HH)$ be a more regular weak solution of (\ref{eq:uc_model}) satisfying additionally
			$$
			\U \in L^8(0,T^\dagger;L^4(\Omega))\cap L^4(0,T^\dagger;H^1(\Omega)), \quad \HH \in L^\infty(0,T^\dagger;L^4(\Omega))\cap L^4(0,T^\dagger;H^1(\Omega)) ,	
			$$
			for some $T^\dagger\leq T$. Let $(\U,\HH)$ start from the initial data $(\U_0,\HH_0)$. Then the  relative energy given by (\ref{eq:rel_en_full}) satisfies the inequality
			\begin{align}
			\mathcal{E}(t) + b\mathcal{D} \leq \mathcal{E}(0)  + \int_0^t g(\dta)\mathcal{E}(\dta) \ddta, \label{eq:relgron}
			\end{align}
			for almost all $t\in(0,T^\dagger)$. Here $g\in L^1(0,T^\dagger)$, and $\mathcal{D}$ is given by (\ref{eq:rel_gron}) and $b>0$.
		\end{Proposition}
	\end{tcolorbox}
	
	\begin{tcolorbox}
		\begin{Theorem}[Weak-strong uniqueness]
			\label{theo:wsu}
			Let the assumptions of Theorem \ref{theo:relative_energy} hold and let the initial data coincide, i.e. $\u_0=\U_0, \CC_0=\HH_0$.
			Then any weak solution $(\u,\CC)$ in the sense of Definition \ref{defn:weak_sol_UC} coincides with the more regular weak solution $(\U,\HH)$  almost everywhere in $\Omega\times(0,T^\dagger)$.
		\end{Theorem}
	\end{tcolorbox}
	\begin{Remark}\phantom{2em}
	\begin{itemize}
	    \item  The above result implies the (local) uniqueness in the class of more regular solutions.
	    \item In two space dimensions by reviewing Theorem \ref{theo:cond-energy-eq}, and the corresponding remark and quick inspection of the arguments used in the proof of Theorem  \ref{theo:relative_energy}, we can conclude the uniqueness of weak solutions for initial data satisfying $\CC_0\in W^{\frac{1}{2}, \frac{4}{3}}_{{\bf n}}(\Omega)$.
	\end{itemize}
	
	\end{Remark}
	
	\section{Existence of weak solutions}
	
	In order to analyse the Peterlin viscoelastic system \eqref{eq:uc_model} we consider several energy type estimates as follows.
	By formally taking the inner product of $(\ref{eq:uc_model})_1$ with $\u$ and $(\ref{eq:uc_model})_2$ with $\trC\mathbf{I}/2$ and integrating over the domain $\Omega$ yields
	\begin{align}
	&\td \(\int_\Omega \frac{1}{2}\snorm*{\u}^2 + \frac{1}{4}\snorm*{\trC}^2 \dx \) \label{eq:cen1} \\
	&+ \int_\Omega \eta\snorm*{\Du}^2 + \frac{\varepsilon}{2}\snorm*{\na\trC}^2 + \frac{1}{2}\(\snorm*{\trC}^4 + a\snorm*{\trC}^3\) - \frac{1}{2}\(\snorm*{\trC}^2 + a\trC\) \dx \leq 0. \nonumber
	\end{align}
	Estimating the last integral of (\ref{eq:cen1}) by the Hölder inequality we find after applying the Gronwall Lemma \ref{lem:gronwall} that
	\begin{align}
	\u \in & \;L^\infty(0,T;L^2(\Omega))\cap L^2(0,T;H^1(\Omega)), \label{eq:uc_reg1}\\
	\trC \in & \; L^\infty(0,T;L^2(\Omega))\cap L^2(0,T;H^1(\Omega))\cap L^4(\Omega_T). \nonumber
	\end{align}
	Since for a smooth solution the matrix $\CC$ is positive definite, we find also that $\CC\in L^4(\Omega_T)$, due to the norm equivalence in Proposition \ref{prop:matineq}.
	Now we can take the Frobenius inner product of $(\ref{eq:uc_model})_2$ with $\CC/2$ and obtain
	\begin{align}
	&\td \(\int_\Omega \frac{1}{4}\snorm*{\CC}^2 \dx \) + \frac{\varepsilon}{2}\int_\Omega \snorm*{\na\CC}^2 + \frac{1}{2}\(\snorm*{\trC\CC}^2 + a\snorm*{\trC}\snorm*{\CC}^2\) \dx \nonumber \\
	&\leq \int_\Omega \trC^2 + a\trC \dx + 2\int_\Omega (\na\u\CC):\CC \dx. \label{eq:cen2}
	\end{align}
	The first integral of (\ref{eq:cen2}) can be treated as in (\ref{eq:cen1}). The second integral of (\ref{eq:cen2}) can be bounded as follows
	\begin{equation}
	2\int_\Omega (\na\u\CC):\CC \dx\leq 2\norm*{\na\u}_2\norm*{\CC}_4^2 \leq \norm*{\na\u}_2^2 + \norm*{\CC}_4^4. \label{eq:cen2_est}
	\end{equation}
	Using again the Gronwall Lemma \ref{lem:gronwall} on \eqref{eq:cen2} with \eqref{eq:cen2_est} yields the  following estimates:
	\begin{equation}
	\CC \in L^\infty(0,T;L^2(\Omega))\cap L^2(0,T;H^1(\Omega)). \label{eq:uc_reg2}
	\end{equation}
	The free energy inequality (\ref{eq:uc_freeenergy}) can be formally derived by taking the inner product of $(\ref{eq:uc_model})_1$ with $\u$ and $(\ref{eq:uc_model})_2$ with $\trC\mathbf{I}/2-\CC^{-1}/2$ and applying (\ref{eq:jacobi}):
	\begin{align}
	\td \Big(&\int_\Omega \frac{1}{2}\snorm*{\u}^2 + \frac{1}{4}\snorm*{\trC}^2 - \frac{1}{2}\trlnC \dx \Big) \label{eq:uc_free1}\\
	+ &\int_\Omega \eta\snorm*{\Du}^2 + \frac{\varepsilon}{2}\snorm*{\na\trC}^2 - \frac{\varepsilon}{2}\na\CC:\na\CC^{-1} \dx \nonumber\\
	+&\frac{1}{2}\int_\Omega \snorm*{\trC}^4 + a\snorm*{\trC}^3 - d\snorm*{\trC}^2 + ad\trC \dx\nonumber \\
	-&\frac{1}{2}\int_\Omega d\snorm*{\trC}^2 + ad\snorm*{\trC} - \trC\mathrm{tr}(\CC^{-1}) - a\mathrm{tr}(\CC^{-1}) \dx \leq 0 . \nonumber
	\end{align}
	To proceed, we first expand the diffusion term involving the inverse matrix by
	\begin{align*}
	\int_\Omega \na\CC::\na\(\CC^{-1}\)\dx &= -\sum_{|\alpha|=1}\int_\Omega \partial^\alpha_x\CC:\CC^{-1}\partial_x^\alpha\CC\CC^{-1}\dx \\
	&= -\sum_{|\alpha|=1}\int_\Omega \tr{ \CC^{-1/2}\partial_x^\alpha\CC\CC^{-1/2}\CC^{-1/2}\partial_x^\alpha\CC\CC^{-1/2} }\dx\\
	&= -\sum_{|\alpha|=1}\norm*{\CC^{-1/2}\partial_x^\alpha\CC\CC^{-1/2}}_2^2 =: -\norm*{\CC^{-1/2}\na\CC\CC^{-1/2}}^2_2\leq -\frac{1}{d}\norm*{\na\mathrm{tr}(\log\CC)}^2_2.
	\end{align*}
	Here we have used the cyclic property of the trace, symmetry of $\CC$, the existence of a square root $\CC^{1/2}$ and Lemma \ref{lem:matinv}. Rewriting (\ref{eq:uc_free1})  yields
	\begin{align}
	&\td \(\int_\Omega \frac{1}{2}\snorm*{\u}^2 + \frac{1}{4}\snorm*{\trC}^2 - \frac{1}{2}\trlnC \dx \) \label{eq:cn3}\\
	&+ \int_\Omega \eta\snorm*{\Du}^2 + \frac{\varepsilon}{2}\snorm*{\na\trC}^2 + \frac{\varepsilon}{6}\snorm*{\na\tr{\log\CC}}^2 \dx \nonumber\\
	&+ \frac{1}{2}\int_\Omega \(\trC^2 + a\trC \)\mathrm{tr}(\TT + \TT^{-1} - 2\I) \dx \leq 0.  \nonumber
	\end{align}
	Since $\CC$ is symmetric positive definite, applying Lemma \ref{lem_logid} and the bounds \eqref{eq:uc_reg1} we obtain from (\ref{eq:cn3}) the additional information
	\begin{align*}
	\trlnC &\in L^\infty(0,T;L^1(\Omega)), \quad \na\mathrm{tr}(\log\CC), \hspace{0.4em}\CC^{-1/2}\na\CC\CC^{-1/2} \in L^2(\Omega_T),\\
	a \trr{\CC^{-1}} &, \hspace{0.4em} \trr{\CC^{-1}}\trC \in L^1(\Omega_T).
	\end{align*}
	The control of $\trr{\log \CC}$ in $L^\infty(0,T;L^1(\Omega))$ can be obtained as follows. The free energy is non-negative and bounded from below by Lemma \ref{lem_logid}. The upper bound follows from the inequality \eqref{eq:cn3}. The velocity and the trace contribution are already bounded in $L^\infty(0,T;L^1(\Omega))$, hence we obtain the desired bound.

	\subsection{Galerkin approximation and a priori estimates}
	The goal of this section is to derive an approximation scheme based on the Galerkin method for the velocity $\u$, see \cite{LukacovaMedvidova.2017}.
	Let $\vv_j, j=1,\ldots,\infty$ be smooth basis functions of $V = \overline{\spn\{\vv_j\}_{j=1}^{\infty}}$. Here the $\vv_j$ are divergence-free and subjected to the homogeneous Dirichlet boundary conditions. Then we define the $m$-th Galerkin approximation of $\u$ by
	\begin{equation}
	\u_m(x,t) = \sum_{j=1}^m g_{jm}(t)\vv_j(x), \quad \u_m(0) = \u_{0m}.
	\end{equation}
	Furthermore, $\CC_m(\u_m)$ denotes the solution of the parabolic problem $(\ref{eq:uc_model})_2$ for $\CC_m$, i.e.
	\begin{align}
	   \int_\Omega \frac{\partial \u_m}{\partial t}\cdot\vv &+ (\u_m\cdot\na)\u_m\cdot\vv + \eta\Du_m:\mathrm{D}\vv - \trr{\CC_m}\CC_m:\na\vv \dx=0, \label{eq:nsm}\\
	   \frac{\partial \CC_m}{\partial t} + (\u_m\cdot\na)\CC &- \na\u_m\CC_m+\CC_m\na\u_m^\top - \bigchi(\trr{\CC_m})\CC_m + \Phi(\trr{\CC_m})\I + \varepsilon\Delta\CC_m =0. \label{eq:cm}
	\end{align}
	Due to standard theory for ordinary differential equations there exists a finite-dimensional approximation of the velocity $\u_m$.
	Uniform bounds imply the existence up to time $T$ for all $m$. Further, parabolic regularity, see \cite{Pruess.2016}, and the bounds on the velocity $\u_m$ show that there is a conformation stress tensor $\CC_m\in C^1((0,T];C^2(\Omega))$.
	Since $\CC_m$ is positive definite for every $m$, see \cite{LukacovaMedvidova.2017}, we obtain the following regularity result by integrating the Galerkin approximations of (\ref{eq:cen1}) and (\ref{eq:cen2}) in time
	\begin{align}
	\u_m & \in_{bdd} L^\infty(0,T;L^2(\Omega))\cap L^2(0,T;H^1(\Omega)), \label{eq:uc_discest1}\\
	\trr{\CC_m}& \in_{bdd} L^\infty(0,T;L^2(\Omega))\cap L^2(0,T;H^1(\Omega))\cap L^4(\Omega_T), \nonumber\\
	\CC_m& \in_{bdd} L^\infty(0,T;L^2(\Omega))\cap L^2(0,T;H^1(\Omega))\cap L^4(\Omega_T).\nonumber
	\end{align}
	
	Using (\ref{eq:uc_discest1}) yields
	\begin{equation}
	\bigchi(\trr{\CC_m})\CC_m  \in_{bdd}  L^{\frac{4}{3}}(\Omega_T), \quad \Phi(\trr{\CC_m})  \in_{bdd} L^2(\Omega_T). \label{eq:uc_discest2}
	\end{equation}
	Here the notation $\in_{bdd}$ means the function family is uniformly bounded in the corresponding space.
	\subsection{Compact embeddings}
	In order to derive suitable integrability of the time derivative we rewrite the Navier-Stokes \eqref{eq:nsm} as operator equation
	\begin{equation}
	\frac{\partial \u_m}{\partial t} = -\mathcal{A}\u_m + \mathcal{B}\u_m + \mathcal{H}\CC_m,
	\end{equation}
	with the operators $\mathcal{A}, \mathcal{B}, \mathcal{H}$ defined by
	\begin{align*}
	\mathcal{A}&:V \to V^{*} & &\inp{\mathcal{A}\u,\vv}:= \int_\Omega \eta \Du:\mathrm{D}\vv,\hspace{4.9em} \vv\in V, \\
	\mathcal{B}&:V \to V^{*} & &\inp{\mathcal{B}\u,\vv}:= -\int_\Omega (\u\cdot\na)\u \cdot \vv,\hspace{3.4em} \vv \in V, \\
	\mathcal{H}&:H^1(\Omega) \to V^{*} & &\inp{\mathcal{H}\CC,\vv}:= -\int_\Omega \tr{\CC}\CC:\na\vv, \hspace{2.0em}\vv \in V.
	\end{align*}
	By using Sobolev embedding, the following estimate holds
	\begin{align}
	\int_0^T \norm*{\frac{\partial \u_m}{\partial t}}_{V^*}^p\dt &\leq c\int_0^T \norm*{\Du_m}_2^p + \norm*{\trr{\CC_m}\CC_m}_2^p + \norm*{\u_m}_3^p\norm*{\na\u_m}_2^p \dt. \label{eq:uc_treg}
	\end{align}
	Using the regularity result (\ref{eq:uc_discest1}) we find that $\frac{\partial\u_m}{\partial t} \in_{bdd} L^{\frac{4}{3}}(0,T;V^*)$, i.e. taking $p=\frac{4}{3}$ in (\ref{eq:uc_treg}).
	
	Next we consider the evolution equation for the conformation tensor \eqref{eq:cm} which can be rewritten as an operator equation of the form
	\begin{equation}
	\frac{\partial \CC_m}{\partial t} + \varepsilon\Delta\CC_m = \mathbf{F}_m \label{eq:uc_operatorC}
	\end{equation}
	with $\mathbf{F}_m:=-(\u_m\cdot\na)\CC_m + (\na\u_m)\CC_m + \CC_m(\na\u_m)^T - \bigchi(\tr{\CC_m})\CC_m + \Phi(\tr{\CC_m})\I $.
	We calculate
	\begin{align}
	\hspace{-1em}\int_0^T\norm*{\mathbf{F}_m}_{H^{-1}}^p\dt \leq &\int_0^T \norm*{\u_m}_3^p\norm*{\na\CC_m}_2^p + \norm*{\CC}_4^p\norm*{\na\u_m}_2^p \nonumber\\
	&+ \norm*{\bigchi(\tr{\CC_m})\CC_m}_{4/3}^p + \norm*{\Phi(\tr{\CC_m})}_2^p \dt. \label{eq:uc_Fspace}
	\end{align}
	Standard calculations using (\ref{eq:uc_discest1}), (\ref{eq:uc_discest2}) show that $\mathbf{F}_m \in_{bdd} L^{\frac{4}{3}}(0,T;H^{-1}(\Omega))$, i.e. $p=\frac{4}{3}$ in (\ref{eq:uc_Fspace}). This implies by bootstrapping that
	\begin{equation*}
	\frac{\partial \CC_m}{\partial t} \in_{bdd} L^{\frac{4}{3}}(0,T;H^{-1}(\Omega)), \quad \CC_m\in_{bdd} L^{\frac{4}{3}}(0,T;H^1(\Omega)). \label{eq:uc_disctd}
	\end{equation*}
	
	Using these estimates and the Aubin-Lions Lemma we have for suitable subsequences the following convergence results
	\begin{align}
	&\u_m \rightharpoonup^{\star} \u \in L^\infty(0,T;L^2(\Omega)), &\CC_m &\rightharpoonup^{\star} \CC \in L^\infty(0,T;L^2(\Omega)), \label{eq:uc_convtable}\\
	&\u_m \rightharpoonup \u \in L^2(0,T;V)\cap L^{\frac{10}{3}}(\Omega_T), &\CC_m &\rightharpoonup \CC \in L^2(0,T;H^1(\Omega))\cap L^{4}(\Omega_T),\nonumber \\
	&\u_m \rightarrow \u  \in L^2(0,T;L^p(\Omega)) \text{ for } p<6, &\CC_m &\rightarrow \CC  \in L^2(0,T;L^p(\Omega)) \text{ for } p<6,\nonumber \\
	&\u_m \rightharpoonup \u \text{ a.e. in } \Omega\times(0,T), &\CC_m &\rightharpoonup \CC \text{ a.e. in } \Omega\times(0,T),\nonumber \\
	&\frac{\partial \u_m}{\partial t} \rightharpoonup \frac{\partial \u}{\partial t} \in L^{\frac{4}{3}}(0,T;H^{-1}(\Omega)), &\frac{\partial \CC_m}{\partial t} &\rightharpoonup \frac{\partial \CC}{\partial t} \in L^{\frac{4}{3}}(0,T;H^{-1}(\Omega)).\nonumber
	\end{align}
	
	Furthermore, considering the Galerkin approximation of the energy inequality (\ref{eq:cn3}) we obtain
	\begin{align}
	\norm*{\trr{\CC_m}^2 - 2\trr{\log\CC_m}}_{L^\infty(L^1)} &+ \norm*{\na\trr{\log\CC_m}}_{L^2(L^2)}+\norm*{\CC_m^{-1/2}\na\CC_m\CC_m^{-1/2}}_{L^2(L^2)} \leq c,\label{eq:uc_freeendisc}\\
	a\norm*{\trr{\CC_m^{-1}}}_{L^1(L^1)} &+  \norm*{\trr{\CC_m}\trr{\CC_m^{-1}}}_{L^1(L^1)} \leq c, \nonumber\\
	\trr{\log\CC_m} &\rightharpoonup \overline{\trr{\log\CC_m}} \in L^2(0,T;H^1(\Omega)).\nonumber
	\end{align}
	
	\subsection{Limit passing}
	In this section we will pass to the limit  in the Galerkin approximation of \eqref{eq:nsm} and \eqref{eq:cm} as $m\to \infty$. Here we will focus on the limiting process in the main nonlinearities of \eqref{eq:nsm} and \eqref{eq:cm}. Let $\varphi\in L^\infty(0,T)$ be a time dependent test function. We start with the Galerkin approximation of the Navier-Stokes equations \eqref{eq:nsm} and consider the elastic stress tensor term
	\begin{align*}
	P_{1,m}&:=\int_0^T\int_\Omega (\trr{\CC_m}\CC_m - \trC\CC):\na\vv\varphi \dx\dt \\
	&= \int_0^T\int_\Omega (\trr{\CC_m-\CC})\CC_m:\na\vv\varphi + \trC(\CC_m-\CC):\na\vv\varphi \dx\dt\\
	&\leq \int_0^T\(\norm*{\CC_m-\CC}_4\norm*{\trC}_4 + \norm*{\trr{\CC_m-\CC}}_4\norm*{\CC_m}_4\)\norm*{\na\vv}_2\norm*{\varphi}_\infty \dt\\
	&\leq c\(\norm*{\CC_m-\CC}_{L^2(L^4)}^2\norm*{\trC}_{L^2(L^4)}^2 + \norm*{\trr{\CC_m-\CC}}_{L^2(L^4)}^2\norm*{\CC_m}_{L^2(L^4)}^2\).
	\end{align*}
	Since $\trr{\CC_m}, \CC_m$ are strongly convergent to $\trC, \CC$ in $L^2(0,T;L^4(\Omega))$, cf. (\ref{eq:uc_convtable}), we get $P_{1,m}\to 0$ as $m\to\infty$.

	We now turn to the conformation tensor equation \eqref{eq:cm} and first consider
	\begin{align}
	P_{2,m}:=\int_0^T\int_\Omega \bigchi(\trr{\CC_m})\CC_m:\DD\varphi(t)\dx\dt. \label{eq:uc_chilim}
	\end{align}
	The integrand of $P_{2,m}$ is bounded in $L^r(\Omega_T)$ for $\frac{1}{r} = \frac{3}{4} + \frac{1}{6}$, which yields $r=\frac{12}{11}>1$. Therefore, by the Vitali Lemma, see \cite{Folland.2011}, (\ref{eq:uc_chilim}) converges to its limit $P_2$ since the integrand is continuous and convergent a.e. in $\Omega_T$, cf. (\ref{eq:uc_convtable}).
	
	Finally, we consider the upper convective derivative
	\begin{align}
	P_{3,m}:=&\int_0^T\int_\Omega \Big[(\na\u_m)\CC_m - (\na\u)\CC + \CC_m(\na\u_m)^T - \CC(\na\u)^T\Big]:\DD\varphi \dx\dt \label{eq:uc_limucm} \\
	= & \int_0^T\int_\Omega  \Big[(\na\u_m-\na\u)\CC_m + \na\u(\CC_m-\CC) + \CC_m(\na\u_m-\na\u)^T \\
	&   \qquad \qquad + (\CC_m-\CC)(\na\u)^T\Big]:\DD\varphi \dx\dt. \nonumber
	\end{align}
	Thanks to the strong convergence of $\CC_m$ in $L^2(0,T;L^4(\Omega))$ and the weak convergence of $\na\u_m$ in $L^2(\Omega_T)$, cf. (\ref{eq:uc_convtable}), $P_{3,m}\to 0$ as $m\to\infty$. The limit process in other terms is standard, cf. \cite{LukacovaMedvidova.2017}.

	\subsection{Energy inequalities and positive definiteness of the conformation tensor}
	In this section we will consider the limiting process in the energy inequalities \eqref{eq:uc_energy}, \eqref{eq:uc_freeenergy}. Furthermore, for the limit in the free energy inequality \eqref{eq:uc_freeenergy} we need to prove that the limiting conformation tensor $\CC$ is positive definite a.e. in $\Omega_T$.\\
	First we consider the limit in the discrete version of (\ref{eq:cen1}), cf.~\eqref{eq:uc_energy}. We observe that due to the convergence given by (\ref{eq:uc_convtable}) we can apply the same arguments as in \cite{Brunk.}. Indeed, applying the Fatous Lemma, we have for a weakly convergent sequence $\{g_m\}_{m=1}^\infty$, $g_m \rightharpoonup^{\star} g$ that
	\begin{align*}
	    \norm{g(t)}_2 \leq \norm{g}_{L^\infty(0,T;L^2(\Omega))}&\leq \liminf_{m\to\infty}\norm{g_m}_{L^\infty(0,T;L^2(\Omega))}, \\
	    \norm{g}_{L^2(0,T;L^2(\Omega))} &\leq \liminf_{m\to\infty}\norm{g_m}_{L^2(0,T;L^2(\Omega))}.
	\end{align*}
Consequently, we derive
	\begin{align}
	&\(\int_\Omega \frac{1}{2}\snorm*{\u(t)}^2 + \frac{1}{4}\snorm*{\trCt}^2 \dx \) +\int_{\Omega_t}\eta\snorm*{\Du}^2 + \frac{\varepsilon}{2}\snorm*{\na\trC}^2 + \frac{1}{2}\( \snorm*{\trC}^4 + a\snorm*{\trC}^3\) \dx\ddta\ \nonumber \\
	&\leq  \frac{1}{2}\int_{\Omega_t} \snorm*{\trC}^2 + a\trC \dx\ddta + \(\int_\Omega \frac{1}{2}\snorm*{\u(0)}^2 + \frac{1}{4}\snorm*{\trCo}^2 \dx \). \label{eq:icen1}
	\end{align}
	Here we have used the strong convergence of $\CC_m$, cf. (\ref{eq:uc_convtable}), to pass to the limit in the first integral on the right-hand side of (\ref{eq:icen1}).

	In what follows we want to prove a similar limit for the Galerkin approximation of (\ref{eq:uc_freeenergy}). Here we follow the ideas in \cite{Barrett.2011,Barrett.2017,Barrett.2012b,Barrett.2018c}. In order to identify the limit correctly we first need to prove the positive definiteness of the limit $\CC$, since all approximations $\CC_m$ are positive definite by construction. Repeating the same calculations yielding to (\ref{eq:cn3}) for the Galerkin approximations we deduce that
	\begin{equation}
	\int_{\Omega_T}\trCim \dx\dt \leq c(a), \label{eq:uc_invc}
	\end{equation}
	where the constant $c(a)$ depends inversely on $a$, i.e. it blows up for $a\to0$. Using the positive definiteness of $\CC_m$  and the estimate (\ref{eq:uc_invc}) we obtain the following estimates on  $\CC_m^{-1}$
	\begin{equation}
	 \int_{\Omega_T} |\CC_m^{-1}| \dx\dt \leq c(a). \label{eq:uc_fullinva}
	\end{equation}
	With these result at hand we can prove the following useful lemma by contradiction.
	\begin{Lemma}
		\label{lema:uc_spd}
		Let $a>0$ and $\CC$ be the limit of the sequence of positive definite solutions $\CC_m$ of (\ref{eq:uc_operatorC}). Then the limit $\CC$ is positive definite a.e. in $\Omega_T$. If $a=0$ we can conclude positive semi-definiteness of the limit solution $\CC$ a.e. in $\Omega_T$.
	\end{Lemma}
	\begin{proof}
		Assume the existence of a set $D\subset\Omega_T$ having non-zero measure  such that $\CC(x,t)$ is not positive definite for $(x,t)\in D$. By construction $\CC$ is the limit of positive definite sequence $\CC_m$ which yields that $\CC$ is positive semi-definite, due to the strong convergence of $\CC_m$ in $L^2(\Omega_T)$, cf. (\ref{eq:uc_convtable}). This already finishes to proof for the case $a=0$. \\
		This implies that $\CC$ has at least one zero eigenvalue in $D$. Thus, there exists a vector function $\vv\in L^\infty(\Omega_T)^d$ such that $\snorm*{\vv}=1$ in $D$ and $\vv=\mathbf{0}$ in $\Omega_T\setminus D$, such that $\vv^T\CC\vv = 0$ a.e. in $\Omega_T$.  We estimate the measure of $D$ by
		\begin{align}
		\snorm*{D} = \int_D \snorm*{\vv} \dx\dt = \int_{\Omega_T} \snorm*{\vv} \dx\dt&= \int_{\Omega_T} \snorm*{\CC_m^{-1/2}\CC_m^{1/2}\vv} \dx\dt\nonumber\\
		&\leq \(\int_{\Omega_T} \snorm*{\CC_m^{-1/2}}^2 \dx\dt\)^{\frac{1}{2}} \(\int_{\Omega_T} \snorm*{\CC_m^{1/2}\vv}^2 \dx\dt\)^{\frac{1}{2}} \nonumber\\
		&\leq \(\int_{\Omega_T} \snorm*{\CC_m^{-1}} \dx\dt\)^{\frac{1}{2}} \(\int_{\Omega_T} \snorm*{\vv^T\CC_m\vv} \dx\dt\)^{\frac{1}{2}} \nonumber\\
		&\leq c(a)\(\int_{\Omega_T} \snorm*{\vv^T\CC_m\vv} \dx\dt\)^{\frac{1}{2}}. \label{eq:uc_limitindef}
		\end{align}
		It is easy to see that if $a>0$ then $c(a)$ is bounded and the right side of the inequality (\ref{eq:uc_limitindef}) converges as $m\to\infty$ since $\CC_m$ converges strongly to $\CC$ in $L^2(\Omega_T)$, cf. (\ref{eq:uc_convtable}). However, by assumption $\vv^T\CC\vv = 0$ a.e. in $\Omega_T$. Consequently, $|D|=0$, which is a contradiction and implies that $\CC$ is positive definite a.e. in $\Omega_T$.
	\end{proof}
	
	In the case $a>0$ we can now pass to the limit in the Galerkin approximation of the free energy inequality (\ref{eq:uc_freeenergy}). 
	Let us rewrite the terms in (\ref{eq:uc_freeendisc}) as follows,
	\begin{equation}
	\trr{\CC_m^{-1}}= h_1 \circ \CC_m,\quad \trr{\CC_m}\trr{\CC_m^{-1}}= h_2 \circ \CC_m,\quad \trr{\log\CC_m}= h_3 \circ \CC_m,
	\end{equation}
	where $h_i: (0,\infty) \to\mathbb{R}, i=1,\ldots,3$ are continuous functions.
	Since $\CC_m,\CC>0$ a.e. in $\Omega_T$ and converge for a.e. $(x,t)$ in $\Omega\times(0,T)$, cf. (\ref{eq:uc_convtable}), it is easy to see that
	\begin{equation}
	h_1 \circ \CC_m \longrightarrow h_1 \circ \CC,\quad h_2 \circ \CC_m \longrightarrow h_2 \circ \CC,\quad h_3 \circ \CC_m \longrightarrow h_3 \circ \CC \text{ a.e. in } \Omega_T \label{eq:uc_aeinvlog}
	\end{equation}
	see, e.g., \cite[Exercise 2.37]{Folland.2011}. Applying (\ref{eq:uc_aeinvlog}) for $(\ref{eq:uc_freeendisc})_2$ we  conclude that $\overline{\trr{\log\CC_m}}=\trr{\log\CC}$ a.e. in $\Omega_T$, i.e.
	\begin{equation}
	\trr{\log\CC_m} \rightharpoonup \trr{\log\CC} \in L^2(0,T;H^1(\Omega)). \label{uc_limtrlog}
	\end{equation}
	We proceed with the terms in $(\ref{eq:uc_freeendisc})_1$. Since $\CC_m$ is positive definite a.e. in $\Omega_T$, so is  $\CC_m^{-1}$.  Consequently, we have $\tr{\CC_m^{-1}}>0$ a.e. in $\Omega_T$. Application of the Fatou Lemma yields
	\begin{equation}
	\iQt \trr{\CC^{-1}} \dx\ddta = \iQt \overline{\trr{\CC_m^{-1}}} \dx\ddta \leq \liminf_{m\to\infty} \iQt \trr{\CC_m^{-1}} \dx\ddta. \label{eq:uc_limc-1}
	\end{equation}
	We have used here $\overline{\trr{\CC_m^{-1}}}=\trr{\CC^{-1}}$ a.e. in $\Omega_T$. Similarly, we derive for a.a. $t\in(0,T)$
	\begin{equation}
	\iQt \trr{\CC^{-1}}\trC \dx\ddta = \iQt \overline{\trr{\CC_m^{-1}}\trr{\CC_m}} \dx\ddta \leq \liminf_{m\to\infty} \iQt \trr{\CC_m^{-1}}\trr{\CC_m} \dx\ddta. \label{eq:uc_limc-1c}
	\end{equation}

	Having obtained the convergences (\ref{eq:uc_convtable}), (\ref{uc_limtrlog}), (\ref{eq:uc_limc-1}), (\ref{eq:uc_limc-1c}) we can pass to the limit in the free energy inequality and obtain for a.a. $t\in(0,T)$
	\begin{align}
	& \(\int_\Omega \frac{1}{2}\snorm*{\u(t)}^2 + \frac{1}{4}\snorm*{\trCt}^2 - \frac{1}{2}\trlnCt\dx \) \label{eq:icn3} \\
	&+ \iQt \eta\snorm*{\Du}^2 + \frac{\varepsilon}{2}\snorm*{\na\trC}^2 + \frac{\varepsilon}{6}\snorm*{\na\tr{\log\CC}}^2 + \frac{1}{2}\(\trC^2 + a\trC \)\mathrm{tr}(\TT + \TT^{-1} - 2\I ) \dx\ddta \nonumber\\
	&\leq \(\int_\Omega \frac{1}{2}\snorm*{\u(0)}^2 + \frac{1}{4}\snorm*{\trCo}^2 - \frac{1}{2}\trlnCo \dx \).  \nonumber
	\end{align}

	\section{Parabolic Regularity and Conditional Energy Equality}
	This section is devoted to the proof of Theorem \ref{theo:cond-energy-eq}.  We shall show that the weak solution $(\u,\CC)$ obtained in Theorem \ref{thm:existence} satisfies the energy equality \eqref{cond-energy-eq} under additional integrability assumption $\CC \in L^{s}(0,T; L^{r}(\Omega))$ with $(s,r)$ satisfying \eqref{Serrin-s-r}.
	
	\medskip
	
	\subsection{Parabolic regularity of $\CC$}\label{sec:para-reg}
	We rewrite equation $\eqref{eq:uc_model}_2$ for each component $\CC_{i,j}, 1\leq i,j\leq 3,$ as
	\baa\label{eq-C-ij-0}
	&\frac{\partial\CC_{i,j} }{\partial t}  -  \varepsilon\Delta\CC_{i,j} = h_{1,i,j} + h_{2,i,j},  \quad &&\mbox{in} \  \Omega_T,\\
	&\p_{\bf n} \CC_{i,j} = 0 \quad &&\mbox{on} \ (0,T)\times \p\Omega, \\
	& \CC_{i,j}(0,\cdot) = (\CC_{0})_{i,j}\quad &&\mbox{in} \ \Omega,
	\eaa
	where
	\baa
	h_{1} : = -(\u\cdot \nabla ) \CC, \quad h_{2}: =  (\nabla\u)\CC + \CC(\nabla\u)^\top + \Phi(\trC)\mathbf{I} -\bigchi(\trC)\CC.
	\eaa
	Using the integrability of a weak solution in Definition \ref{defn:weak_sol_UC}, and applying H\"older's inequality gives
	\bee\label{est-h-ij-1}
	h_{1} \in L^{1}(0,T;L^{\frac{3}{2}}(\Omega))\cap  L^{2}(0,T;L^{1}(\Omega)), \quad h_{2} \in L^{\frac{4}{3}}(\Omega_T).
	\eee
	By interpolation, we derive
	\bee\label{est-h-ij-2}
	h_{1} \in L^{\tilde s}(0,T;L^{\tilde r}(\Omega)), \quad \frac{2}{\tilde s} + \frac{3}{\tilde r} = 4, \quad 1\leq \tilde s \leq 2, \ 1\leq \tilde r \leq \frac{3}{2}.
	\eee
	In particular, choosing $\tilde s = \tilde r$ gives $h_{1} \in L^{\frac{4}{3}}(0,T;L^{\frac{6}{5}}(\Omega))$. Hence, for each $1\leq i,j\leq 3$, the source term in \eqref{eq-C-ij-0} satisfies $h_{1,i,j} + h_{2,i,j} \in  L^{\frac{4}{3}}(0,T;L^{\frac{6}{5}}(\Omega))$. Recall that the initial datum
	$$
	\CC_{0} \in W^{\frac{1}{2}, \frac{4}{3}}_{{\bf n}} (\Omega) \subset W^{\frac{1}{2}, \frac{6}{5}}_{{\bf n}} (\Omega).
	$$
	Thus, applying Lemma \ref{lem-parabolic-1} implies that the unique solution $\CC_{i,j}$ to \eqref{eq-C-ij-0} satisfies
	\baa\label{2nd-time-est-C-1}
	\CC\in L^{\frac{4}{3}}(0,T;W^{2,\frac{6}{5}}(\Omega)), \quad \p_{t} \CC\in L^{\frac{4}{3}}(0,T;L^{\frac{6}{5}}(\Omega)).
	\eaa
	
	\medskip
	
	On the other hand, we can decompose this unique solution $\CC_{i,j} = \CC_{i,j}^{(1)} + \CC_{i,j}^{(2)}$ where $\CC_{i,j}^{(1)}$ and $\CC_{i,j}^{(2)}$ solves
	\baa\label{eq-C-ij-1}
	&\frac{\partial\CC_{i,j}^{(1)}}{\partial t}  -  \varepsilon\Delta\CC_{i,j}^{(1)}= h_{1,i,j} ,  \quad &&\mbox{in} \ \Omega_T,\\
	&\p_{\bf n} \CC_{i,j}^{(1)} = 0 \quad &&\mbox{on} \ (0,T)\times \p\Omega, \\
	& \CC_{i,j}^{(1)}(0,x) = 0\quad &&\mbox{in} \ \Omega
	\eaa
	and
	\baa\label{eq-C-ij-2}
	&\frac{\partial\CC_{i,j}^{(2)} }{\partial t}  -  \varepsilon\Delta\CC_{i,j}^{(2)}= h_{2,i,j} ,  \quad &&\mbox{in} \ \Omega_T,\\
	&\p_{\bf n} \CC_{i,j}^{(2)} = 0 \quad &&\mbox{on} \ (0,T)\times \p\Omega, \\
	& \CC_{i,j}^{(2)}(0,\cdot) = (\CC_{0})_{i,j} \quad &&\mbox{in} \ \Omega,
	\eaa
	respectively.  Using \eqref{est-h-ij-1}, \eqref{est-h-ij-2} and applying Lemma \ref{lem-parabolic-1} to \eqref{eq-C-ij-1}, \eqref{eq-C-ij-2} gives
	\baa\label{est-C-ij-1}
	\CC_{i,j}^{(1)}\in L^{\tilde s}(0,T;W^{2,\tilde r}(\Omega)), \quad \frac{\p \CC_{i,j}^{(1)}}{\p t}\in  L^{\tilde s}(0,T;L^{\tilde r}(\Omega)),
	\eaa
	and
	\baa\label{est-C-ij-2}
	\CC_{i,j}^{(2)}\in L^{\frac{4}{3}}(0,T;W^{2,\frac{4}{3}}(\Omega)), \quad \frac{\p \CC_{i,j}^{(2)}}{\p t}\in  L^{\frac{4}{3}}(\Omega_T),
	\eaa
	where $(\tilde s, \tilde r)$ satisfies \eqref{est-h-ij-2} except the board-line cases $\tilde s = 1, \tilde r = \frac{3}{2}$ and $\tilde s = 2, \tilde r =1$. Consequently, we have completed the proof of the first part of Theorem \ref{theo:cond-energy-eq}.

	\subsection{Conditional energy equality}
	We proceed with the proof of the energy equality \eqref{cond-energy-eq} under additional integrability assumption $\CC \in L^{s}(0,T; L^{r}(\Omega))$ for some $(s,r)$ satisfying \eqref{Serrin-s-r} i.e.
	\baa\label{Serrin-s-r-recall}
	\frac{2}{s} + \frac{3}{r} \leq 1, \quad 2<s<\infty,  \ 3 < r <\infty.
	\eaa
	
	\medskip

	Recall the main result of Section \ref{sec:para-reg}: $\CC = \CC^{(1)} + \CC^{(2)}$ with
	\baa\label{est-C-ij-1-re}
	\CC^{(1)}\in L^{\tilde s}(0,T;W^{2,\tilde r}(\Omega)), \quad \frac{\p \CC^{(1)}}{\p t}\in  L^{\tilde s}(0,T;L^{\tilde r}(\Omega)),
	\eaa
	and
	\baa\label{est-C-ij-2-re}
	\CC^{(2)}\in L^{\frac{4}{3}}(0,T;W^{2,\frac{4}{3}}(\Omega)), \quad \frac{\p \CC^{(2)}}{\p t}\in  L^{\frac{4}{3}}(\Omega_T),
	\eaa
	for all $(\tilde s, \tilde r)$ satisfying
	$$\frac{2}{\tilde s} + \frac{3}{\tilde r} = 4, \quad  1 < \tilde s < 2, \ 1 < \tilde r < \frac{3}{2}.$$
	It is easy to see that the Lebesgue conjugate numbers $(\tilde s', \tilde r')$ satisfy
	$$
	\frac{2}{\tilde s'} + \frac{3}{\tilde r'} = 1, \quad \quad  2 < \tilde s' < \infty,  \ 3 <  \tilde r' < \infty,
	$$
	which is a subcase of \eqref{Serrin-s-r-recall}. Thus,
	$$
	\CC \in L^{4}(\Omega_T) \cap L^{s}(0,T; L^{r}(\Omega)) \subset \big( L^{\frac{4}{3}}(\Omega_T) +  L^{\tilde s}(0,T;L^{\tilde r}(\Omega))\big)'.
	$$
	Together with the energy estimates for $\u$ and $\CC$ in \eqref{defn:weak_sol_UC}, one can verify by a density argument that $\CC$ can be chosen as a test function in \eqref{eq:weak_sol_UC-2}. This gives for each $t\in (0,T]$ that
	\baa\label{eq:weak_sol_UC-2-C}
	& \int_{0}^{t} \int_{\Omega} \CC :  \frac{\p \CC}{\p t} \dx \ddta - \int_{0}^{t} \int_{\Omega} (\u\cdot\nabla)\CC : \CC \dx \ddta - \int_{0}^{t} \int_\Omega \e |\nabla \CC |^{2} \dx \ddta  \\
	& =  \int_{0}^{t}\int_\Omega \bigchi(\trC)\CC:\CC \dx\ddta - \int_{0}^{t} \int_\Omega \Phi(\trC) {\rm tr}\,\CC \dx \ddta  \\
	& \quad +  \int_{\Omega} |\CC (t)|^{2}\dx -   \int_{\Omega} |\CC_{0}|^{2}\dx.
	\eaa
	We next claim that for a.a. $t\in (0,T)$
	\baa\label{C-C-1}
	& \int_{0}^{t} \int_{\Omega} \CC :  \frac{\p \CC}{\p t} \dx \ddta  = \frac{1}{2}\int_{\Omega} |\CC (t)|^{2}\dx  - \frac{1}{2}\int_{\Omega} |\CC_{0}|^{2}\dx, \\
	& \int_{0}^{t} \int_{\Omega} (\u\cdot\nabla)\CC : \CC \dx \ddta = 0.
	\eaa
	
	To prove \eqref{C-C-1}, we employ the classical Friedrichs mollification in the spatial variable.  Let  $\phi_{\de}$ be a standard Friedrichs mollifier in $\R^{d}$ with $ \de\in (0,1)$ small. Define $\CC_{\de}(t, x): = \phi_{\de} \ast \CC (t, x)$. For each $t\in [0,T]$, here we treat $\CC(t,\cdot)$ as a function that is extended in $W^{1,2}_{0}(\Omega_{1})$ where $\Omega_{1}: = \{x\in \R^{d}; {\rm dist}\,(x,\p \Omega) <1\}$. It is easy to verify that for each $0<\delta < 1$, $\CC_{\de}(t, \cdot) \in C_{c}^{\infty}(\Omega_{2})$ with $\Omega_{2}: = \{x\in \R^{d}; {\rm dist}\,(x,\p \Omega) < 2\}$ that
	$$
	\| \p^{\alpha}_{x}\CC_{\de}(t, \cdot)\|_{L^{\infty}(\R^{d})} \leq  c(\alpha, \delta) \|\CC(t, \cdot)\|_{L^{1}(\Omega)}, \quad \text{ for all } \alpha \in \mathbb{N}^{3}.
	$$
	
	\medskip

	We start by showing the second equality in \eqref{C-C-1}. We have
	\baa\label{C-C-1.5}
	\int_{0}^{t} \int_{\Omega} (\u\cdot\nabla)\CC : \CC \dx \ddta & = \int_{0}^{t} \int_{\Omega} (\u\cdot\nabla)\CC : (\CC-\CC_{\delta}) \dx \ddta  + \int_{0}^{t} \int_{\Omega} (\u\cdot\nabla)(\CC-\CC_{\delta}) : \CC_{\delta} \dx \ddta \\
	& +  \int_{0}^{t} \int_{\Omega} (\u\cdot\nabla)\CC_{\delta} : \CC_{\delta} \dx \ddta.
	\eaa
	
	We first consider the last term in \eqref{C-C-1.5}. Since $\CC_{\delta} (t, \cdot)\in C_{c}^{\infty}(\R^{d})$ and $\u(t, \cdot) \in W^{1,2}_{0}(\Omega), \, \div \u =0$ for a.a. $t\in (0,T)$, we obtain
	\baa\label{C-C-2}
	\int_{\Omega} (\u\cdot\nabla)\CC_{\delta} : \CC_{\delta} \dx = \int_{\Omega} (\u\cdot\nabla) \frac{1}{2}|\CC_{\delta}|^{2} \dx = -\int_{\Omega} \div \u \frac{1}{2} |\CC_{\delta}|^{2} \dx = 0.
	\eaa
	
	\medskip
	
	We then consider the second term in \eqref{C-C-1.5}
	\baa\label{C-C-1.6}
	\int_{0}^{t} \int_{\Omega} (\u\cdot\nabla)(\CC-\CC_{\delta}) : \CC_{\delta} \dx \ddta & = \int_{0}^{t} \int_{\Omega} \big((\u\cdot\nabla)\CC -((\u\cdot \nabla)\CC)_{\delta}\big) : \CC_{\delta} \dx\ddta\\
	& + \int_{0}^{t} \int_{\Omega} \big(((\u\cdot \nabla)\CC)_{\delta} - (\u \cdot \nabla) \CC_{\delta}\big) : \CC_{\delta} \dx\ddta .
	\eaa
	In \eqref{est-h-ij-2}  we have shown that
	\bee\label{est-h-ij-2-recall}
	(\u\cdot\nabla)\CC \in L^{\tilde s}(0,T;L^{\tilde r}(\Omega)) \quad \mbox{for all $(\tilde s, \tilde r)$ such that} \ \frac{2}{\tilde s} + \frac{3}{\tilde r} = 4, \ 1\leq \tilde s \leq 2, \ 1\leq \tilde r \leq \frac{3}{2}.
	\eee
	Then for each $(\tilde s, \tilde r)$ satisfying \eqref{est-h-ij-2-recall},  for a.e. $t\in (0,T)$, there holds
	\baa
	& \|(\u\cdot\nabla)\CC(t, \cdot) - ((\u\cdot\nabla)\CC)_{\de} (t, \cdot)\|_{\tilde r}\to 0, \ \ \mbox{as $\de \to 0$},\\
	& \|(\u\cdot\nabla)\CC(t, \cdot) - ((\u\cdot\nabla)\CC)_{\de} (t, \cdot)\|_{\tilde r}\leq 2 \|(\u\cdot\nabla)\CC(t, \cdot) \|_{\tilde r} \in L^{\tilde s}(0,T).
	\nn
	\eaa
	Lebesgue's dominated convergence theorem implies that
	\baa\label{C-de-cov1}
	\|(\u\cdot\nabla)\CC - ((\u\cdot\nabla)\CC)_{\de}\|_{L^{\tilde s}(L^{\tilde r})} \to 0,  \quad \mbox{as $\de \to 0$}.
	\eaa

	We then obtain
	\baa \label{C-C-1.7}
	&\int_{0}^{t} \int_{\Omega} \big((\u\cdot\nabla)\CC -((\u\cdot \nabla)\CC)_{\delta}\big) : \CC_{\delta} \dx\ddta \\
	&\leq  \| (\u\cdot\nabla)\CC - ((\u\cdot\nabla)\CC)_{\de}\|_{L^{s'}(L^{r'})} \|\CC_{\delta}\|_{L^{s}(L^{r})}  \\
	& \leq c \|(\u\cdot\nabla)\CC - ((\u\cdot\nabla)\CC)_{\de}\|_{L^{\tilde s}(L^{\tilde r})}  \|\CC\|_{L^{s}(L^{r})},
	\eaa
	which converges to $0$ as $\de\to 0$. Here we have used the fact that there exists $(\tilde s, \tilde r)$ satisfying  \eqref{est-h-ij-2-recall} such that
	$$
	s'\leq \tilde s, \quad r' \leq \tilde r.
	$$
	
	\medskip
	
	To show the convergence of  the second term in \eqref{C-C-1.6}, we introduce the Friedrichs' commutator lemma, see, e.g.,  Lemma 6.7 in \cite{Novotny.2004}.
	\begin{Lemma}\label{lem-Fre-comm} Suppose that $d \geq 2$. Let $1 < q, \beta < \infty$ and $\frac{1}{q} + \frac{1}{\beta} \leq 1.$ Let $1 \leq  \alpha \leq  \infty$ and $\frac{1}{\alpha} + \frac{1}{p} \leq 1$. Suppose that
		$$
		\rho \in L^{\alpha}(0,T; L_{loc}^{\beta}(\R^{d})), \quad u \in L^{p}(0,T; W^{1,q}_{loc}(\R^{d})).
		$$
		Then
		$$
		\| (u \cdot \nabla \rho)_{\de} - (u \cdot \nabla \rho_{\de })\|_{L^{s}(L^{r}_{loc} (\R^{d}))} \to 0  \ \ \mbox{as $\de\to 0$}
		$$
		with
		$$
		\frac{1}{s} = \frac{1}{\alpha} + \frac{1}{p} , \quad \frac{1}{r} = \frac{1}{q} + \frac{1}{\beta}.
		$$
		Moreover, $u \cdot \nabla \rho  = \div{\rho u} - \rho \div u$.
	\end{Lemma}
	
	Since $\u \in L^{2}(0,T; W^{1,2}(\Omega))$ and $\CC \in L^{4}(\Omega_T)$,  Lemma \ref{lem-Fre-comm} yields
	$$
	\|((\u\cdot \nabla)\CC)_{\delta} - (\u \cdot \nabla) \CC_{\delta}\|_{L^{\frac{4}{3}}(L^{\frac{4}{3}})} \to 0  \ \ \mbox{as $\de\to 0$}.
	$$
	Together with the uniform boundedness
	$$
	\| \CC_{\de} \|_{L^{4}(L^{4})}  = \left\|   \| \CC_{\de} \|_4 \right\|_{L^{4}(0,T)} \leq  \left\|   \| \CC \|_{4} \right\|_{L^{4}(0,T)} \leq \| \CC \|_{L^{4}(L^{4})} \leq  c,
	$$
	we deduce
	\baa\label{C-C-3}
	\int_{0}^{t} \int_{\Omega} \big(((\u\cdot \nabla)\CC)_{\delta} - (\u \cdot \nabla) \CC_{\delta}\big) : \CC_{\delta} \dx\ddta \to 0.
	\eaa

	\medskip

	For a.e. $t\in (0,T)$ there holds
	\baa
	& \| \CC(t, \cdot) - \CC_{\de} (t, \cdot)\|_r\to 0 \ \ \mbox{as $\de \to 0$},\\
	& \| \CC(t, \cdot) - \CC_{\de} (t, \cdot)\|_r\leq 2 \| \CC(t, \cdot) \|_r \in L^{s}(0,T).
	\eaa
	Lebesgue's dominated convergence theorem implies that
	\baa\label{C-de-cov2}
	\|\CC - \CC_{\de} \|_{L^{s}(L^{r})} \to 0  \quad \mbox{as $\de \to 0$}.
	\eaa
	Together with the estimates $(\u \cdot \nabla) \CC \in {L^{s'}(0,T;L^{r'}(\Omega)}$, applying H\"older's inequality implies
	\baa\label{C-C-4}
	\int_{0}^{t} \int_{\Omega} (\u\cdot\nabla)\CC : (\CC-\CC_{\delta}) \dx \ddta  \to 0  \quad \mbox{as $\de \to 0$}.
	\eaa
	
	\medskip
	
	By \eqref{C-C-1.5}, \eqref{C-C-2}, \eqref{C-C-1.6}, \eqref{C-C-1.7}, \eqref{C-C-3} and \eqref{C-C-4}, we obtain $\eqref{C-C-1}_{2}$.

	\medskip
	
	Now we consider the first equality in \eqref{C-C-1}.  Similarly, as \eqref{C-de-cov1} and \eqref{C-de-cov2}  we can show
	\baa\label{C-de-cov3}
	\|\CC - \CC_{\de} \|_{L^{4}(L^4)} \to 0, \quad  \norm*{ \frac{\p \CC}{\p t} - \frac{\p \CC_{\de}}{\p t}  }_{L^{\frac{4}{3}}(L^{\frac{4}{3}}) + L^{s'}(L^{r'})}\to 0.
	\eaa
	This together with \eqref{C-de-cov2} implies
	\baa\label{C-de-0}
	\int_{0}^{t} \int_{\Omega} \CC :  \frac{\p \CC}{\p t} \dx \ddta = \lim_{\de\to 0}  \int_{0}^{t} \int_{\Omega} \CC_{\de} :  \frac{\p \CC_{\de}}{\p t} \dx \ddta.
	\eaa
	A standard density argument with respect to the time variable implies
	\baa\label{t-deri-C}
	\frac{\p |\CC_{\de}|^{2}}{\p t}  = 2  \CC_{\de} :  \frac{\p \CC_{\de}}{\p t},
	\eaa
	where $\frac{\p |\CC_{\de}|^{2}}{\p t} $ stands for the weak time derivative of $|\CC_{\de}|^{2}$.   Define
	$$
	F_{\de} (t) := \frac 12 \int_{\Omega}   |\CC_{\de}(t) |^{2} \dx.
	$$
	We have the following lemma for $F_{\de}$:
	\begin{Lemma}\label{lem:F-de-deriva} The function $F_{\de} \in L^{\infty}(0,T)$ and $F_{\de}$ admits a weak derivative $F_{\de}' \in L^{\frac{4}{3}}(0,T)$. Moreover,
		$$
		F_{\de}'(t) = \frac 12 \int_{\Omega}\frac{\p |\CC_{\de}|^{2}}{\p t} \dx =  \int_{\Omega} \CC_{\de} :  \frac{\p \CC_{\de}}{\p t}\dx.
		$$
	\end{Lemma}
	
	\begin{proof}[Proof of Lemma \ref{lem:F-de-deriva}]  Let $\phi_{n}(x)\in C_{c}^{\infty}(\Omega)$ such that $$
		\phi_{n} = 1 \ \mbox{on} \ \left\{x\in \Omega; {\rm dist} (x, \p \Omega) > \frac{1}{n}\right\}.$$  Clearly
		$$
		\phi_{n} \to 1 \ \mbox{in} \ L^{r}(\Omega), \ \mbox{for all $1\leq r <\infty$}.
		$$
		Given $\psi\in C^{\infty}_{c}(0,T)$ we compute
		\baa\label{F-de-1}
		\int_{0}^{T} F_{\de} (t)  \psi'(t) \dt  & = \frac 12 \int_{0}^{T} \int_{\Omega}   |\CC_{\de}(t)|^{2} \dx \, \psi'(t) \dt\\
		& =  \frac 12 \int_{0}^{T} \int_{\Omega}   |\CC_{\de}(t)|^{2} \psi'(t) \phi_{n}(x)\dx\dt +  \frac 12 \int_{0}^{T} \int_{\Omega}   |\CC_{\de}(t)|^{2} \psi'(t) (1-\phi_{n}(x))\dx\dt.
		\eaa
		Direct calculation gives
		\baa\label{F-de-2}
		&\frac 12 \int_{0}^{T} \int_{\Omega}   |\CC_{\de}(t)|^{2} |\psi'(t) | |(1-\phi_{n}(x))| \dx\dt\\
		& \leq \|\CC_{\de}\|_{L^{4}(L^4)}^{2} \|\psi'\|_{L^{\infty}(0,T)} \|1-\phi_{n}\|_{2} \to 0, \quad \mbox{as $n\to \infty.$}
		\eaa
		By \eqref{t-deri-C} we can deduce
		\baa\label{F-de-3}
		\frac 12 \int_{0}^{T} \int_{\Omega}   |\CC_{\de}(t)|^{2} \psi'(t) \phi_{n}(x)\dx\dt & =  -\frac 12 \int_{0}^{T} \int_{\Omega}   \frac{\p |\CC_{\de}|^{2}}{\p t}  \psi(t) \phi_{n}(x)\dx\dt \\
		& = - \int_{0}^{T} \int_{\Omega}  \CC_{\de} :  \frac{\p \CC_{\de}}{\p t} \psi(t) \phi_{n}(x)\dx\dt \\
		& = - \int_{0}^{T} \int_{\Omega}  \CC_{\de} :  \frac{\p \CC_{\de}}{\p t} \psi(t) \dx\dt \\
		& \quad - \int_{0}^{T} \int_{\Omega}  \CC_{\de} :  \frac{\p \CC_{\de}}{\p t} \psi(t) (\phi_{n}(x)-1)\dx\dt.
		\eaa
		Since $\CC \in L^{\infty}(0,T;L^{2}(\Omega))$ and $\frac{\partial \CC}{\partial t} \in L^{\frac{4}{3}}(\Omega_T) + L^{2}(0,T;L^{1}(\Omega))$, we obtain
		$$
		\CC_{\de} \in L^{\infty}(0,T; C^{\infty}_{c}(\R^{d})), \quad \frac{\partial \CC_{\de}}{\partial t} \in L^{\frac{4}{3}}(0,T;C^{\infty}_{c}(\R^{d})) .
		$$

		Thus,
		\baa\label{F-de-4}
		&\int_{0}^{T} \int_{\Omega} | \CC_{\de} |  \left| \frac{\p \CC_{\de}}{\p t} \right|  |\psi(t)| | (\phi_{n}(x)-1)|\dx\dt \\
		&\leq  \|\CC_{\de}\|_{L^{\infty}(\Omega_T)} \left \| \frac{\p \CC_{\de}}{\p t} \right \|_{L^{\frac{4}{3}}(L^{\frac{4}{3}})}  \|\psi\|_{L^{\infty}(0,T)}  \|\phi_{n} - 1\|_4 \to 0, \quad \mbox{as $n\to \infty.$}
		\eaa
		
		Applying \eqref{F-de-1}--\eqref{F-de-4} implies
		\baa\label{F-de-5}
		\int_{0}^{T} F_{\de} (t)  \psi'(t) \dt   = - \int_{0}^{T} \left(\int_{\Omega}  \CC_{\de} :  \frac{\p \CC_{\de}}{\p t} \dx \right)\psi(t)  \dt.
		\eaa
		Together with the fact $\CC_{\de} \in L^{\infty}(0,T; C^{\infty}_{c}(\R^{d}))$ and $\frac{\p \CC_{\de}}{\p t}  \in L^{\frac{4}{3}}(0,T; C^{\infty}_{c}(\R^{d}))$, we finally deduce
		\baa\label{F-de-6}
		F_{\de} '(t)    =  \int_{\Omega}  \CC_{\de} :  \frac{\p \CC_{\de}}{\p t} \dx  \in L^{\frac{4}{3}}(0,T).
		\eaa
		
	\end{proof}

	Lemma \ref{lem:F-de-deriva} implies that $F_{\de} \in W^{1,\frac{4}{3}}(0,T)$. The Sobolev embedding implies that $F_{\de}\in C^{0,\frac{1}{4}}([0,T])$ is H\"older continuous.  Consequently, it is rather easy to verify by density argument that for each $t\in [0,T]$ there holds
	\baa\label{F-t-0-1}
	\int_{0}^{t} F_{\de}'(\tau)\ddta = F_{\de}(t) - F_{\de}(0).
	\eaa
	Since $\CC\in C_{w}([0,T];L^{2}(\Omega))$ then for a.e. $t\in (0,T)$, we have
	\baa\label{F-t-0-2}
	F_{\de}(0) = \frac 12 \int_{\Omega}   |\CC_{0,\de}|^{2} \dx \to \frac 12 \int_{\Omega}   |\CC_{0}|^{2} \dx, \quad F_{\de}(t) = \frac 12 \int_{\Omega}   |\CC_{\de}(t)|^{2} \dx \to \frac 12 \int_{\Omega}   |\CC(t)|^{2} \dx.
	\eaa
	Hence, by \eqref{C-de-0}, \eqref{t-deri-C}, Lemma \ref{lem:F-de-deriva}, \eqref{F-t-0-1} and  \eqref{F-t-0-2}, we finally deduce $\eqref{C-C-1}_{1}$. We  thus finished the proof of \eqref{C-C-1}. As a consequence, together with \eqref{eq:weak_sol_UC-2-C}, we obtain the equality \eqref{cond-energy-eq} and complete the proof of Theorem \ref{theo:cond-energy-eq}.

	\section{Relative Energy}
	
	\subsection{Expansion of the relative energy}
	The goal of this section is to expand the relative energy (\ref{eq:rel_en_full}) by using the corresponding energy inequalities and inserting more regular weak solutions as  suitable test functions.
	We recall that (\ref{eq:uc_energy}) can be written in the following abstract form
	\begin{align}
	E(\u,\CC)(0) &+ \int_0^t D(\u,\CC) \ddta \leq E(\u,\CC)(0) + \int_0^t F(\CC) \ddta, \label{eq:absenergy} \\
	D(\u,\CC) &= \int_\Omega \eta\snorm*{\Du}^2 + \frac{\varepsilon}{2}\snorm*{\na\trr{\CC-\HH}}^2 + \frac{1}{2}\bigchi(\trC)\trC^2 \dx, \\
	\mathcal{F}(\CC) &= \int_\Omega\frac{1}{2}\Phi(\trC)\trC \dx. \label{eq:absenergy2}
	\end{align}
	
	To this end we expand the relative energy (\ref{eq:rel_en_full}) as follows
	\begin{align*}
	\mathcal{E}(t) &= E(\u,\CC)(t) + E(\U,\HH)(t) - R(\u,\CC\vert\U,\HH)(t),\\
	\mathcal{E}(t)-\mathcal{E}(0)	&= -\int_0^t D(\u,\CC) - F(\CC) + D(\U,\HH)- F(\HH) \ddta  - \int_0^t \frac{\partial}{\partial t}R(\u,\CC\vert\U,\HH) \ddta.
	\end{align*}
	Here we used the energy inequality/equality, respectively and the reminder is defined by
	\begin{equation*}
	R(\u,\CC\vert\U,\HH) = \int_\Omega \u\cdot\U + \frac{1}{2}\trC\trH  - \frac{1}{2}\snorm*{\CC-\HH}^2 \dx.
	\end{equation*}
	
	Next, we will expand the time derivative of the reminder
	\begin{align}
	- \int_0^t \frac{\partial}{\partial t}R(\u,\CC\vert\U,\HH) \ddta =& \overbrace{ -\iQt   \U\cdot\frac{\partial }{\partial t}\u + \u\cdot\frac{\partial }{\partial t}\U \dx\ddta}^{I_1} \\
	& \overbrace{-\frac{1}{2}\iQt \tr{\frac{\partial }{\partial t}\CC}\trH + \trC\tr{\frac{\partial }{\partial t}\HH}\dx\ddta}^{I_2} \label{eq:rel_ex_exp1} \\
	&  \underbrace{+\frac{1}{2}\iQt \frac{\partial }{\partial t}\snorm*{\CC-\HH}^2 \dx\ddta}_{I_3}. \label{eq:rel_ex_exp2}
	\end{align}
	
	First we focus on the terms $I_1$ and $I_2$. The term $I_3$ will be treated separately. In what follows we will apply the Hölder, the Young and interpolation inequalities repeatedly.
	
	For the term $I_1$ we insert $\vv=\U$ in weak formulation \eqref{eq:weak_sol_UC-1} for $(u,\CC)$ and $\vv=\u$ into a weak formulation for $(\U,\CC)$. Note that by density and the assumptions for this theorem, these are valid test functions,  and we obtain
	\begin{align}
	I_1 =  \iQt &(\u\cdot\na)\u\cdot\U + (\U\cdot\na)\U\cdot\u + 2\eta\Du:\mathrm{D}\U \nonumber \\
	&+ \trC\CC:\na\U + \trH\HH:\na\u \dx\ddta.\label{eq:I1}
	\end{align}
	
	We start with the viscosity term by considering $I_1$ and the velocity gradient terms from $D(\u,\CC)$, $D(\U,\HH)$:
	\begin{align}
	P_1:=\iQt  -\eta\snorm*{\Du}^2 - \eta\snorm*{\DU}^2 +  2\eta\Du:\mathrm{D}\U \dx\ddta= -\iQt \eta \snorm*{\Du-\DU}^2 \dx\ddta. \label{eq:P1}
	\end{align}
	
	Next we deal with the convective terms. Applying integration by parts and adding the zero $-((\u-\U)\cdot\na)\U\cdot\U$ in the third line below yields
	\begin{align}
	 P_2:=&\iQt (\u\cdot\na)\u\cdot\U + (\U\cdot\na)\U\cdot\u \dx\ddta = \iQt (\u\cdot\na)\u\cdot\U - (\U\cdot\na)\u\cdot\U \dx\ddta \nonumber  \\
	 =& \iQt((\u-\U)\cdot\na)\u\cdot\U - ((\u-\U)\cdot\na)\U\cdot\U \dx\ddta\nonumber\\
	 =& \iQt((\u-\U)\cdot\na)(\u-\U)\cdot\U \dx\ddta\nonumber\\
	 \leq &\,c\int_0^t \norm*{\u-\U}_4\norm*{\na\u-\na\U}_2\norm*{\U}_4 \ddta  \nonumber\\
	\leq  &\,\delta \int_0^t \norm*{\Du-\DU}_2^2  + c(\delta)\int_0^t  \norm*{\U}_4^8 \mathcal{E}_1(\u,\CC\vert\U,\HH)(\dta)  \ddta.\label{eq:P2}
	\end{align}
	
	Note that we have used the regularity $\U\in L^8(0,T;L^4(\Omega))$. Here and hereafter $\delta, c(\delta)>0$ are suitable constants following from Young's inequality.
	The rest terms
	\begin{equation}
	\iQt  \trC\CC:\na\U + \trH\HH:\na\u \dx\ddta\label{eq:uc_ucmpart}
	\end{equation}
	will be treated together with the trace part of the relative energy later.
	
	For the term $I_2$ we insert $\DD=\CC$ in weak formulation \eqref{eq:weak_sol_UC-2} for $(u,\CC)$ and $\DD=\HH$ into a weak formulation for $(\U,\CC)$. Note that by density and the assumptions for this theorem, these are valid test functions, and we find
	\begin{align}
	I_2 =& \iQt \frac{1}{2}\u\cdot\na\trC\trH + \frac{1}{2}\U\cdot\na\trH\trC - \trr{\CC\na\u}\trH + \trr{\HH\na\U}\trC \nonumber\\
	&+ \frac{1}{2}\bigchi(\trC)\trC\trH +\frac{1}{2} \bigchi(\trH)\trC\trH -\frac{1}{2} \Phi(\trC)\trH \nonumber \\
	&- \frac{1}{2}\Phi(\trH)\trC+ \varepsilon\na\trC\cdot\na\trH \dx\ddta. \label{eq:I2}
	\end{align}

	First we consider the last term of $I_2$ and the tensor gradient term from $D(\u,\CC), D(\U,\HH)$ to obtain
	\begin{equation}
	P_3:=-\frac{\varepsilon}{2}\iQt \snorm*{\na\trC}^2 + \snorm*{\na\trH}^2 - 2 \na\trC\cdot\na\trH \dx\ddta= -\frac{\varepsilon}{2}\iQt \snorm*{\na\trr{\CC-\HH}}^2 \dx\ddta. \label{eq:P3}
	\end{equation}

We continue with the convective terms
	\begin{align}
	P_4:=& \frac{1}{2}\iQt \u\cdot\na\trC\trH + \U\cdot\na\trH\trC \dx\ddta= \frac{1}{2}\iQt (\U-\u)\cdot\na\trH\trC \dx\ddta\nonumber\\
	= &\frac{1}{2}\iQt (\U-\u)\cdot\na\trH\trr{\CC-\HH} \dx\ddta= -\frac{1}{2}\iQt (\U-\u)\cdot\na\trr{\CC-\HH}\trH \dx\ddta\nonumber\\
	\leq& \delta\norm*{\na\trr{\CC-\HH}}_{L^2(L^2)}^2 +  \delta\norm*{\na\u-\na\U}_{L^2(L^2)}^2 + c(\delta)\int_0^t \norm*{\trH}_4^8\mathcal{E}_1(\u,\CC\vert\U,\HH)(\dta)\ddta. \label{P4}
	\end{align}
	Here we have used the regularity $\trH\in L^8(0,T;L^4(\Omega))$.
	Next, the relaxation terms and the corresponding terms from $D(\u,\CC), D(\U,\HH)$ yield
	\begin{align}
	P_5:=&\frac{1}{2} \iQt -\bigchi(\trC)\trC^2 -\bigchi(\trH)\trH^2 +\bigchi(\trC)\trC\trH +\bigchi(\trH)\trC\trH \dx\ddta\nonumber \\
	= &\frac{1}{2}\iQt -\bigchi(\trC)\trr{\CC-\HH}^2  - (\bigchi(\trC)-\bigchi(\trH))\trH\trr{\CC-\HH}\dx\ddta\nonumber \\
	= & \frac{1}{2}\iQt -\bigchi(\trC)\trr{\CC-\HH}^2 - \iQt \bigchi^\prime(\zeta)\trH\trr{\CC-\HH}^2\dx\ddta \label{eq:P5}\\
	\leq&  -\frac{1}{2}\iQt \bigchi(\trC)\trr{\CC-\HH}^2 \dx\ddta + \delta\norm*{\na\trr{\CC-\HH}}_{L^2(L^2)}^2 \\
	 &\qquad+ c(\delta)\int_0^t \norm*{\trH}_4^4\norm*{\bigchi^\prime(\zeta)}_4^4\mathcal{E}_1(\u,\CC\vert\U,\HH)(\dta)\ddta\nonumber\\
	\leq&  -\frac{1}{2}\iQt \bigchi(\trC)\trr{\CC-\HH}^2 \dx\ddta + \delta\norm*{\na\trr{\CC-\HH}}_{L^2(L^2)}^2 \\
	&\qquad+ c(a,\delta)\norm*{\trH}_{L^\infty(L^4)}^4\int_0^t \norm*{\trr{\CC+\HH}}_4^4\mathcal{E}_1(\u,\CC\vert\U,\HH)(\dta)\ddta.  \nonumber
	\end{align}
	
	Here $\zeta$ denotes the convex combination of $\trC,\trH$ from the mean-value theorem. In the last integral we have used the regularity $\trH\in L^\infty(0,T;L^4(\Omega))$.
	
	We focus on the second relaxation terms and the corresponding terms  from $F(\CC), F(\HH)$ and find
	\begin{align}
	P_6:= &\frac{1}{2}\iQt  \Phi(\trC)\trC + \Phi(\trH)\trH  - \Phi(\trC)\trH - \Phi(\trH)\trC \dx\ddta\nonumber \\
	=& \frac{1}{2}\iQt (\Phi(\trC)-\Phi(\trH))\trr{\CC-\HH}\dx\ddta \nonumber\\
	=& \frac{1}{2}\iQt  \trr{\CC-\HH}^2 \dx\ddta\leq c\int_0^t\mathcal{E}_1(\u,\CC\vert\U,\HH)(\dta)\ddta.  \label{eq:P6}
	\end{align}

	Finally, we estimate the term corresponding to the upper convected derivative of $I_2$, cf. (\ref{eq:I2})  and the coupling terms to the Navier-Stokes equations from $I_1$, cf. (\ref{eq:uc_ucmpart}):
	\begin{align}
	P_7:= &\iQt  \trC\CC:\na\U + \trH\HH:\na\u - \tr{\CC\na\u}\trH - \tr{\HH\na\U}\trC \dx\ddta \nonumber\\
	=& \iQt \trr{\CC-\HH}\CC:\na\u - \trr{\CC-\HH}\HH:\na\U\dx\ddta\nonumber \\
	=& \iQt \trr{\CC-\HH}\CC:(\na\u-\na\U) + \trr{\CC-\HH}(\CC-\HH):\na\U \dx\ddta\nonumber\\
	\leq& \delta\norm*{\Du-\DU}_{L^2(L^2)}^2 + \delta\norm*{\na\trr{\CC-\HH}}_{L^2(L^2)}^2  + \delta\norm*{\na(\CC-\HH)}_{L^2(L^2)}^2 \nonumber\\
	&\qquad+ c(\delta)\int_0^t (\norm*{\CC}_r^s + \norm*{\na\U}_2^4)\mathcal{E}(\u,\CC\vert\U,\HH)(\dta)  \ddta.\label{eq:P7}
	\end{align}
	The regularity $\na\U\in L^4(0,T;L^2(\Omega))$ and $\CC\in L^s(0,T;L^r(\Omega))$ was used.
	To obtain the above estimate we have applied
	\begin{equation*}
	\int_0^t\norm*{\na\u -\na\U}_2\norm*{\trr{\CC-\HH}}_p\norm*{\CC}_q \ddta\leq \delta\norm*{\Du -\DU}_{L^2(L^2)}^2  + c(\delta)\int_0^t \norm*{\trr{\CC-\HH}}_p^2\norm*{\CC}_q^2 \ddta
	\end{equation*}
	for $\frac{1}{p}+\frac{1}{q}=\frac{1}{2}$, $p <6 $. Next we interpolate $\norm*{\trr{\CC-\HH}}_p$ between $L^2$ and $L^6$ with a given $\theta\in(0,1)$. Embedding of $L^6$ into $H^1$ and applying the Young inequality yields
	\begin{equation*}
	c(\delta)\int_0^t \norm*{\trr{\CC-\HH}}_p^2\norm*{\CC}_q^2 \ddta\leq \delta\norm*{\na\trr{\CC-\HH}}_{L^2(L^2)}^2 + c(\delta)\int_0^t \norm*{\trr{\CC-\HH}}_2^2\norm*{\CC}_{\frac{3}{1-\theta}}^{\frac{2}{\theta}} \ddta.
	\end{equation*}
	Setting $r=\frac{3}{1-\theta}$ yields  $\norm*{\CC}_{\frac{3}{1-\theta}}^{\frac{2}{\theta}}=\norm*{\CC}_r^{\frac{2r}{r-3}}$. But this is exactly the case when $\frac{2}{s}+\frac{3}{r}=1$ for $s\in(0,\infty), r\in(3,\infty)$.

	Note that $P_{7}$ is the only term in the trace-part where we have to estimate $\CC-\HH$ in the Frobenius norm instead of the trace norm.
	
	Summing up the estimates (\ref{eq:P1})-(\ref{eq:P2}) and  (\ref{eq:P3})-(\ref{eq:P7}) while using the notation of (\ref{eq:relen12}) yields
	\begin{align}
	\mathcal{E}_1(\u,\CC\vert\U,\HH)(t) &  + (1-3\delta)\int_0^t\mathcal{D}_1(\u,\CC\vert\U,\HH)(\tau)\ddta \leq  \mathcal{E}_1(\u,\CC\vert\U,\HH)(0)\label{eq:gronwal_e1}\\
	& \qquad\qquad\qquad\qquad+ c\int_0^t g_1(\tau)\mathcal{E}(\u,\CC\vert\U,\HH)(\dta)\ddta+ \delta\norm*{\na(\CC-\HH)}_{L^2(L^2)}^2,    \nonumber\\
	\mathcal{D}_1(\u,\CC\vert\U,\HH) &=  \eta\norm*{\Du-\DU}_2^2 + \frac{\varepsilon}{2}\norm*{\na\tr{\CC-\HH}}_2^2 +  \frac{1}{2}\norm*{\sqrt{\bigchi(\trC)}\tr{\CC-\HH}}_2^2, \nonumber \\
	g_1(\dta) &= 1 +  \norm*{\CC}_r^s + \norm*{\na\U}_2^4  + \norm*{\trH}_{L^\infty(L^4)}^4\norm*{\trr{\CC+\HH}}_4^4 + \norm*{\trH}_4^8  + \norm*{\U}_4^8. \nonumber
	\end{align}
	
	\subsection{Frobenius Energy}
	In this section we deal with the term $I_3$ arising from the expansion of the relative energy. It should be noted that the structure of the estimates are quite similar to the trace part of the equations.
	Since the weak solutions for $\CC$ is regular enough, direct calculations yield
	\begin{align}
	I_3=&\iQt \td \frac{1}{2}\snorm*{\CC-\HH}^2 \dx\ddta \nonumber\\
	=& \iQt  \tr{\frac{\partial }{\partial t}(\CC-\HH)(\CC-\HH)} \dx\ddta\nonumber \\
	=& 2\iQt \tr{\frac{\partial }{\partial t}\CC\cdot(\CC-\HH)} + \tr{\frac{\partial }{\partial t}\HH\cdot(\HH-\CC)} \dx\ddta \nonumber \\
	=& \iQt -(\u\cdot\na)\CC:(\CC-\HH) + [\na\u\CC + \CC\na\u^\top]:(\CC-\HH) \nonumber\\
	& - \bigchi(\trC)\CC:(\CC-\HH) + \Phi(\trC)\trr{\CC-\HH} - \varepsilon\na\CC:\na(\CC-\HH) \dx\ddta\nonumber\\
	& + \iQt -(\U\cdot\na)\HH:(\HH-\CC) + [\na\U\HH + \HH\na\U^\top]:(\HH-\CC)  \nonumber\\
	&- \bigchi(\trH)\HH:(\HH-\CC) + \Phi(\trH)\trr{\HH-\CC} - \varepsilon\na\HH:\na(\HH-\CC) \dx\ddta.\label{eq:I3}
	\end{align}
	
	In the above computations we have employed Theorem \ref{theo:cond-energy-eq} and inserted $\DD=\HH$ as test function in the weak formulation \eqref{eq:weak_sol_UC-2} for $\CC$ and $\DD=\CC$ as test function in the weak formulation \eqref{eq:weak_sol_UC-2} for $\HH$.
	
	We treat the above terms pairwise and start with the diffusive terms:
	\begin{align}
	Q_1:=-\iQt \varepsilon\na\CC:\na(\CC-\HH) + \varepsilon\na\HH:\na(\HH-\CC) \dx\ddta= - \varepsilon\iQt \snorm*{\na(\CC-\HH)}^2 \dx\ddta.\label{eq:Q1}
	\end{align}

	The convective terms can be estimated as
	\begin{align}
	Q_2:=&\iQt - (\u\cdot\na)\CC:(\CC-\HH)  -  (\u\cdot\na)\HH:(\HH-\CC) \dx\ddta\nonumber\\
	=& -\iQt [(\u\cdot\na)\CC:(\CC-\HH) -  (\U\cdot\na)\HH]:(\CC-\HH) \dx\ddta \nonumber\\
	=& -\iQt [(\u\cdot\na)(\CC-\HH):(\CC-\HH) -  ((\U-\u)\cdot\na)\HH]:(\CC-\HH) \dx\ddta\nonumber\\
	=&  -\iQt  ((\U-\u)\cdot\na)(\CC-\HH):\HH \dx\ddta \nonumber\\
	\leq& \delta\norm*{\na(\CC-\HH)}_{L^2(L^2)}^2 + \delta\norm*{\Du-\DU}_{L^2(L^2)}^2 +  c(\delta)\int_0^t \norm*{\HH}_4^8 \mathcal{E}(\u,\CC\vert \U,\HH)(\dta)\ddta. \label{eq:Q2}
	\end{align}
	Here we have applied the regularity $\HH\in L^8(0,T;L^4(\Omega))$.

    	The terms arising from the upper convected derivative  can be bounded in the following way:
	\begin{align}
	Q_3:=&\iQt  [\na\u\CC + \CC\na\u^\top]:(\CC-\HH) + [\na\U\HH + \HH\na\U^\top]:(\HH-\CC)\dx\ddta \nonumber \\
	=& \iQt [\na\u\CC + \CC\na\u^\top  -  \na\U\HH - \HH\na\U^\top  ]:(\CC-\HH) \dx\ddta \nonumber\\
	=& \iQt \trr{ \na\u\CC\cdot(\CC-\HH) - \na\U\HH\cdot(\CC-\HH) } \dx\ddta\nonumber \\
	=& \iQt \trr{  \na(\u-\U)\CC(\CC-\HH)  + \na\U(\CC-\HH)^2 } \dx\ddta \label{eq:Q3}\\
	\leq& \delta\norm*{\Du-\DU}_{L^2(L^2)}^2 + \delta\norm*{\na(\CC-\HH)}_{L^2(L^2)}^2    + c(\delta)\int_0^t (\norm*{\na\U}_2^4 + \norm*{\CC}_r^s) \mathcal{E}_2(\CC\vert \HH)(\dta) \ddta.  \nonumber
	\end{align}
	This estimate can be obtained following the same idea as in (\ref{eq:P7}). Again we need the regularity $\na\U\in L^4(0,T;L^2(\Omega))$ and $\CC\in L^s(0,T;L^r(\Omega))$.
	Next we focus on the relaxation terms. We start with
	\begin{align}
	Q_4:=&\iQt - \bigchi(\trC)\CC:(\CC-\HH) - \bigchi(\trH)\HH:(\HH-\CC) \dx\ddta\nonumber \\
	=& \iQt - \bigchi(\trC)\snorm*{\CC-\HH}^2 - (\bigchi(\trC) - \bigchi(\trH))\HH:(\CC-\HH)\dx\ddta \nonumber \\
	=& \iQt - \bigchi(\trC)\snorm*{\CC-\HH}^2 \dx\ddta +\iQt \bigchi^\prime(\zeta)\HH:(\CC-\HH)\trr{\CC-\HH}\dx\ddta \nonumber \\
	\leq& -\iQt \bigchi(\trC)\snorm*{\CC-\HH}^2 \dx\ddta  +c(\delta)\int_0^t (\norm*{\bigchi^\prime(\zeta)}_4^4\norm*{\HH}_4^4)\mathcal{E}_2(\CC|\HH)(\dta)\ddta  \nonumber\\
	\leq & -\iQt \bigchi(\trC)\snorm*{\CC-\HH}^2 \dx\ddta + \delta\norm*{\na(\CC-\HH)}_{L^2(L^2)}^2 \nonumber\\
	& \qquad+c(\delta,a)\int_0^t \norm*{\trr{\CC+\HH}}_4^4\norm*{\HH}_4^4\mathcal{E}_2(\CC|\HH)(\dta)\ddta\nonumber\\
		\leq & -\iQt \bigchi(\trC)\snorm*{\CC-\HH}^2 \dx\ddta + \delta\norm*{\na(\CC-\HH)}_{L^2(L^2)}^2 \nonumber\\
		&\qquad +c(\delta,a)\norm*{\HH}_{L^\infty(L^4)}^4\int_0^t \norm*{\trr{\CC+\HH}}_4^4\mathcal{E}_2(\CC|\HH)(\dta)\ddta. \label{eq:Q4}
	\end{align}
	Here $\zeta$ denotes the convex combination of $\trC, \trH$ from the mean-value theorem. Again the regularity $\HH\in L^\infty(0,T;L^4(\Omega))$ has been used.
	
	The second relaxation term is controlled by
	\begin{align}
	Q_5:=&\iQt  \Phi(\trC)\trr{\CC-\HH} + \Phi(\trH)\trr{\HH-\CC}\dx\ddta \nonumber \\
	=& \iQt (\Phi(\trC)-\Phi(\trH))\trr{\CC-\HH} \dx\ddta\nonumber \\
	=& \iQt \trr{\CC-\HH}^2\dx\ddta \leq  c\int_0^t \mathcal{E}_2(\CC|\HH)(\dta)\ddta. \label{eq:Q5}
	\end{align}

	Summing up the estimates (\ref{eq:Q1})-(\ref{eq:Q5}) yields
	\begin{align}
	\mathcal{E}_2(\CC\vert\HH)(t) & + (1-3\delta)\int_0^t\mathcal{D}_2(\CC\vert\HH)(\tau) \ddta\leq \mathcal{E}_2(\CC\vert\HH)(0)  + c\int_0^t\mathcal{E}(\u,\CC\vert\U,\HH)(\dta)g_2(\tau)\ddta  \nonumber \\
	&  \qquad\qquad\qquad\qquad\qquad \qquad +2\delta\norm*{\Du-\DU}_2^2 ,\nonumber\\
	\mathcal{D}_2(\CC\vert\HH) &= \norm*{\sqrt{\bigchi(\trC)}(\CC-\HH)}^2_2 + \varepsilon\norm*{\na(\CC-\HH)}_2^2,\nonumber  \\
	g_2(\dta) & = 1 + \norm*{\trr{\CC+\HH}}_4^4\norm*{\HH}_{L^\infty(L^4)}^4 + \norm*{\na\U}_2^4 + \norm*{\CC}_r^s.\label{eq:gronwallfrob}
	\end{align}

	\subsection{Gronwall-type argument}
	In order to obtain the desired result we have to combine estimates (\ref{eq:gronwal_e1}) and (\ref{eq:gronwallfrob}) to obtain the full relative energy inequality
	\begin{align}
	\mathcal{E}(\u,\CC\vert\U,\HH)(t) &+ (1-6\delta)\int_0^t \mathcal{D}(\u,\CC\vert\U,\HH)(\tau)\ddta \leq \mathcal{E}(\u,\CC\vert\U,\HH)(0) + c\int_0^t g(\dta)\mathcal{E}(\u,\CC\vert\U,\HH)(\dta)\ddta,\nonumber\\
	\mathcal{D}(\u,\CC\vert\U,\HH) &= \eta\norm*{\Du-\DU}_2^2 + \frac{\varepsilon}{2}\norm*{\na\tr{\CC-\HH}}_2^2 +  \frac{1}{2}\norm*{\sqrt{\bigchi(\trC)}\tr{\CC-\HH}}_2^2 \nonumber \\
	&\;\;+  \norm*{\sqrt{\bigchi(\trC)}\CC-\HH}^2_2 + \varepsilon\norm*{\na\CC-\na\HH}_2^2, \nonumber \\
	g(\dta) &\leq 1 + \norm*{\na\U}_2^4  + \norm*{\HH}_{L^\infty(L^4)}^4\norm*{\trr{\CC+\HH}}_4^4  + \norm*{\U}_4^8  + \norm*{\CC}_r^s.   \label{eq:rel_gron}
	\end{align}
	
	Since $\mathcal{D}$ is by construction non-negative, $b$ is positive for $\delta$ small enough and $g\in L^1(0,T^\dagger)$ we can apply the Gronwall Lemma \ref{lem:gronwall} and obtain
	\begin{align}
	\mathcal{E}(\u,\CC\vert\U,\HH)(t) \leq \mathcal{E}(\u,\CC\vert\U,\HH)(0)\exp\(c\int_0^t g(\dta) \ddta\), \label{eq:wsc}
	\end{align}
	which concludes to proof of Theorem \ref{theo:relative_energy}. Observe that for $\u_0=\U_0, \CC_0=\HH_0$ it holds that $\mathcal{E}(\u,\CC\vert\U,\HH)(0)=0$, cf. (\ref{eq:relen_prop}). We obtain Theorem \ref{theo:wsu}.
	
	\section{Relative Energy and Numerical Convergence}
	In this section we illustrate an application of the relative energy inequality (\ref{eq:rel_gron}) in numerical analysis. The relative energy is an appropriate distance to measure the convergence of numerical schemes. 
	In this paper we apply the Lagrange-Galerkin finite element method, that is based on piecewise linear approximations per element. We work here with tetrahedral meshes. The material derivative is approximated by means of characteristics, \cite{LukacovaMedvidova.2017a,LukovMedvidov2017b}. Specifically, our computational domain $\Omega=[0,1]^3$, is triangulated uniformly with a mesh size $h$. The main ingredients of the method in \cite{LukacovaMedvidova.2017a,LukovMedvidov2017b} are the following
	\begin{itemize}
		\item $\mathbb{P}_1$ elements for $\u, p, \CC$.
		\item Brezzi-Pitkäranta stabilization for the pressure $p$.
		\item Fixed-point iteration to solve a coupled nonlinear system for $(\u, p ,\CC)$.
	\end{itemize}
	
	 We recall that in three space dimensions the relative energy is given by
	\begin{align*}
	\mathcal{E}(\u,\CC\vert\U,\HH):= \int_\Omega \frac{1}{2}\snorm*{\u-\U}^2 + \frac{1}{4}\snorm*{\trr{\CC-\HH}}^2  + \frac{1}{2}\snorm*{\CC-\HH}^2 \dx.
	\end{align*}

	In order to measure the convergence order of the Lagrange-Galerkin method we fix the numerical solution with the highest resolution to be the reference solution $\u_{\textrm{ref}}, \CC_{\textrm{ref}}$. Now the error can be computed as
	\begin{equation*}
	\mathcal{E}_{h}(\u_h,\CC_h\vert \u_{\textrm{ref}}, \CC_{\textrm{ref}}).
	\end{equation*}
	In order to measure the rate of convergence we compute the experimental error of convergence (EOC) by
	\begin{align*}
	\mathrm{EOC}_h= \log_2\( \frac{\mathcal{E}_{h}(\u_h,\CC_h\vert \u_{\textrm{ref}}, \CC_{\textrm{ref}})}{\mathcal{E}_{h/2}(\u_{h/2},\CC_{h/2}\vert \u_{\textrm{ref}}, \CC_{\textrm{ref}})} \).
	\end{align*}
	\begin{Remark}
		In our situation the relative energy satisfies the properties of a norm. However, in general the relative energy does not need to satisfy symmetry and the triangle inequality. Hence, one has to be careful when computing the EOC.
	\end{Remark}

	\textbf{Experiment:} We choose the following initial data $$\u_1(0)=\u_2(0)=\u_3(0)=\sin(2\pi x)\sin(2\pi y)\sin(2\pi z), \CC_0=\frac{1}{\sqrt{3}}\I.$$ The model parameters are set to  $\eta=2 ,\varepsilon = 1, a= 0$ with final time $T=1$.  \\[0.5em]
	Our extensive numerical experiments for $d=2,3$ confirm that the numerical solutions are positive definite even for $a=0$. Furthermore, at least experimentally the free energy decreases in time.

	\begin{figure}[H]
		\centering
		\includegraphics[trim={3.cm 0.0cm 3.cm 0.0cm},clip,scale=0.45]{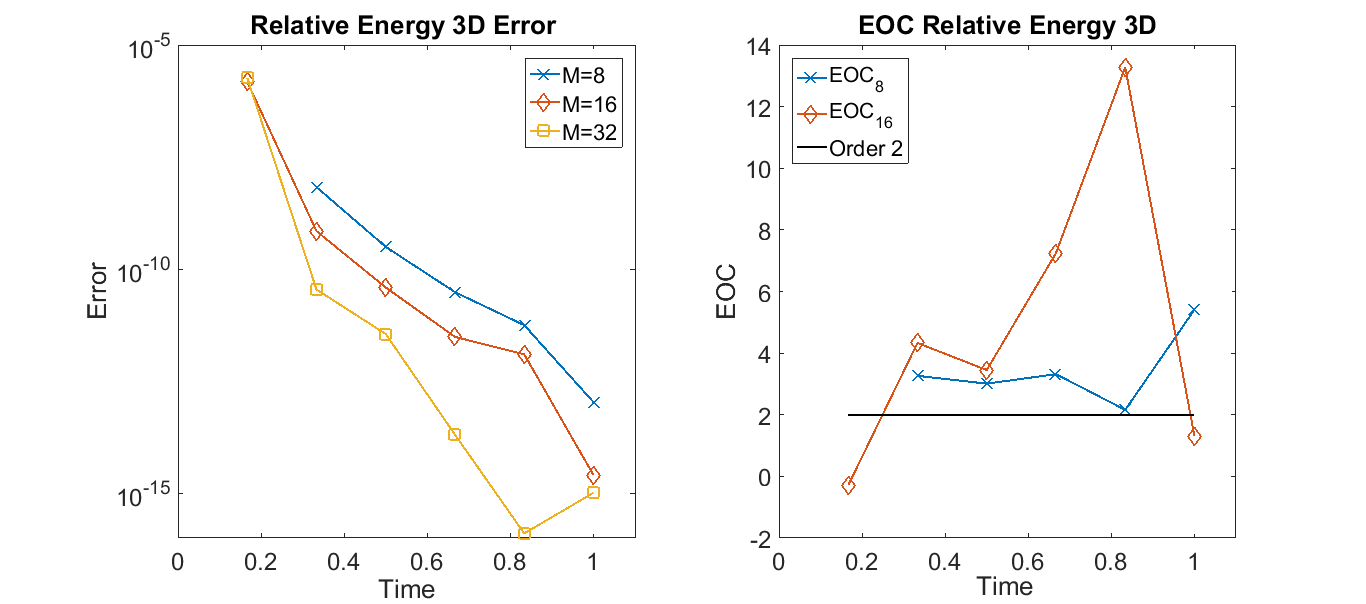}
		\caption{Experiment: Time evolution of the relative energy on a hierarchy of meshes(left) and the convergence order(right). $M$ is the number of points per side, i.e. it relates to $h$.}
		\label{fig:eoc3d}
	\end{figure}
	The first observation is that experimentally the scheme is converging. As time evolves the convergence order of approximately 2 is achieved as expected for squared norms, see also our recent works on the error analysis of the Lagrange-Galerkin method using more standard tools \cite{LukacovaMedvidova.2017a,LukovMedvidov2017b}.

	\section{Conclusion}
	In this work we have proven the existence of global weak solutions for the Peterlin viscoelastic system (\ref{eq:uc_model}) in three space dimensions, see Theorem \ref{thm:existence}. Our proof is based on a combination of the Galerkin method for the incompressible Navier-Stokes equations and a semigroup approach for the time evolution of the conformation tensor. This approach allows to deduce the positive-(semi) definiteness of the tensor $\CC$ which is necessary from the physical point of view. In fact, in  two space  dimensions it is possible to prove existence of weak solutions without showing positive-(semi) definiteness. On the other hand in three space dimensions this property is a crucial ingredient of the existence proof of the conformation tensor.

	Moreover, we provided a regularity result, see Theorem~\ref{theo:cond-energy-eq}, applying methods from the semigroup theory. Using the latter result allows us to apply the relative energy method, see Theorem \ref{theo:relative_energy} which ultimately provides the weak-strong uniqueness result, see Theorem~\ref{theo:wsu}. The relative energy method has wide applications. We illustrate its use as a suitable metric in the experimental convergence study of the Lagrange-Galerkin method. In the future, it will be interesting to analyse theoretically the convergence of Lagrange-Galerkin method using the relative energy.
	
	\section*{Acknowledgement}
	This research of A.B. and M.L. was supported by the German Science Foundation (DFG) under the Collaborative Research Center TRR~146 Multiscale Simulation Methods for Soft Matters (Project~C3). M.L. gratefully acknowledges the support of the Gutenberg Research College fellowship of
the University Mainz. The research of  Y.L. has been supported by the Recruitment Program of Global Experts of China. We would like to thank P.~Tolksdorf, M.~Bachmayr, B.~She and H.~Egger for fruitful discussions on the topic.

	\bibliography{literature_dots}

\begin{thebibliography}{10}

\bibitem{Barrett.2011}
J.~W. Barrett and S.~Boyaval.
\newblock {Existence and approximation of a (regularized) Oldroyd-B model}.
\newblock {\em {Math. Mod. Meth. Appl. Sci.}}, 21(09):1783--1837, 2011.

\bibitem{Barrett.2018}
J.~W. Barrett and S.~Boyaval.
\newblock {Finite element approximation of the FENE-P model}.
\newblock {\em {IMA J. Numer. Anal.}}, 38(4):1599--1660, 2018.

\bibitem{Barrett.2017}
J.~W. Barrett, Y.~Lu, and E.~S{\"u}li.
\newblock {Existence of large-data finite-energy global weak solutions to a
  compressible Oldroyd-B model}.
\newblock {\em {Commun. Math. Sci.}}, 15(5):1265--1323, 2017.

\bibitem{Barrett.2011b}
J.~W. Barrett and E.~S{\"u}li.
\newblock {Existence and equilibration of global weak solutions to finitely
  extensible nonlinear bead-spring chain models for dilute polymers}.
\newblock {\em {Math. Mod. Meth. Appl. Sci.}}, 21(06):1211--1289, 2011.

\bibitem{Barrett.2012}
J.~W. Barrett and E.~S{\"u}li.
\newblock {Existence and equilibration of global weak solutions to Hookean-type
  bead-spring chain models for dilute polymers}.
\newblock {\em {Math. Mod. Meth. Appl. Sci.}}, 22(05):1150024, 2012.

\bibitem{Barrett.2012b}
J.~W. Barrett and E.~S{\"u}li.
\newblock {Finite element approximation of finitely extensible nonlinear
  elastic dumbbell models for dilute polymers}.
\newblock {\em {ESAIM Math. Model. Numer. Anal.}}, 46(4):949--978, 2012.

\bibitem{Barrett.2016b}
J.~W. Barrett and E.~S{\"u}li.
\newblock {Existence of global weak solutions to compressible isentropic
  finitely extensible bead-spring chain models for dilute polymers}.
\newblock {\em {Math. Mod. Meth. Appl. Sci.}}, 26(03):469--568, 2016.

\bibitem{Barrett.2016}
J.~W. Barrett and E.~S{\"u}li.
\newblock {Existence of global weak solutions to compressible isentropic
  finitely extensible nonlinear bead--spring chain models for dilute polymers:
  The two-dimensional case}.
\newblock {\em {J. Differ. Equ.}}, 261(1):592--626, 2016.

\bibitem{Barrett.2018b}
J.~W. Barrett and E.~S{\"u}li.
\newblock {Existence of global weak solutions to the kinetic Hookean dumbbell
  model for incompressible dilute polymeric fluids}.
\newblock {\em {Nonlinear Anal.-Real}}, 39:362--395, 2018.

\bibitem{Barrett.2018c}
J.~W. Barrett and E.~S{\"u}li.
\newblock {Existence of large-data global-in-time finite-energy weak solutions
  to a compressible FENE-P model}.
\newblock {\em {Math. Mod. Meth. Appl. Sci.}}, 28(10):1929--2000, 2018.

\bibitem{Bathory.2020}
M.~Bathory, M.~Bulíček, and J.~Málek.
\newblock Large data existence theory for three-dimensional unsteady flows of
  rate-type viscoelastic fluids with stress diffusion.
\newblock {\em Adv. Nonlinear Anal.}, 10(1):501 -- 521, 2021.

\bibitem{Berselli.2002}
L.~C. Berselli and G.~P. Galdi.
\newblock {Regularity criteria involving the pressure for the weak solutions to
  the Navier-Stokes equations}.
\newblock {\em {Proc. Amer. Math. Soc.}}, 130(12):3585--3595, 2002.

\bibitem{Bird1}
R.~Bird, C.~Curtiss, R.~Armstrong, and O.~Hassager.
\newblock {\em Dynamics of polymeric liquids. Vol. 1: Fluid mechanics}.
\newblock Wiley, 1 edition, 1987.

\bibitem{Bird2}
R.~Bird, C.~Curtiss, R.~Armstrong, and O.~Hassager.
\newblock {\em Dynamics of Polymeric Liquids, Vol. 2: Kinetic Theory}.
\newblock Wiley, 2 edition, 1987.

\bibitem{Brunk.}
A.~Brunk and M.~Luk{\'a}{\v{c}}ov{\'a}-Medvid'ov{\'a}.
\newblock {Global existence of weak solutions to the two-phase viscoelastic
  phase separation: Part I Regular Case}.
\newblock https://arxiv.org/abs/1907.03480 (submitted).

\bibitem{Chemin2001}
J.-Y. Chemin and N.~Masmoudi.
\newblock About lifespan of regular solutions of equations related to
  viscoelastic fluids.
\newblock {\em {SIAM} J. Math. Anal.}, 33(1):84--112, 2001.

\bibitem{Chen2006}
Y.~Chen and P.~Zhang.
\newblock The global existence of small solutions to the incompressible
  viscoelastic fluid system in 2 and 3 space dimensions.
\newblock {\em Commun. Partial Differ. Equ.}, 31(12):1793--1810, 2006.

\bibitem{Chupin.2018}
L.~Chupin.
\newblock {Global strong solutions for some differential viscoelastic models}.
\newblock {\em {SIAM J. Appl. Math.}}, 78(6):2919--2949, 2018.

\bibitem{Constantin.2012}
P.~Constantin and M.~Kliegl.
\newblock {Note on global regularity for two-dimensional Oldroyd-B fluids with
  diffusive stress}.
\newblock {\em {Arch. Ration. Mech. An.}}, 206(3):725--740, 2012.

\bibitem{Feireisl2012}
E.~Feireisl, B.~J. Jin, and A.~Novotn{\'{y}}.
\newblock Relative entropies, suitable weak solutions, and weak-strong
  uniqueness for the compressible {N}avier{\textendash}{S}tokes system.
\newblock {\em J. Math. Fluid Mech.}, 14(4):717--730, 2012.

\bibitem{FERNANDEZCARA2002}
E.~Fernandez-Cara, F.~Guillen, and R.~Ortega.
\newblock Mathematical modeling and analysis of viscoelastic fluids of the
  {O}ldroyd kind.
\newblock {\em Handb. Numer. Anal.}, 8:543--660, 2002.

\bibitem{Folland.2011}
G.~B. Folland.
\newblock {\em {Real analysis: Modern techniques and their applications}}.
\newblock {Pure A. Math.} {John Wiley {\&} Sons}, New York, 2011.

\bibitem{Geissert2012}
M.~Geissert, D.~G\"{o}tz, and M.~Nesensohn.
\newblock ${L}^p$-theory for a generalized nonlinear viscoelastic fluid model
  of differential type in various domains.
\newblock {\em Nonlinear Anal.-Theor.}, 75(13):5015--5026, 2012.

\bibitem{Guillope.1990}
C.~Guillop{\'e} and J.~C. Saut.
\newblock {Existence results for the flow of viscoelastic fluids with a
  differential constitutive law}.
\newblock {\em {Nonlinear Anal.-Theor.}}, 15(9):849--869, 1990.

\bibitem{Gwiazda2018}
P.~Gwiazda, M.~Luk{\'{a}}{\v{c}}ov{\'{a}}-Medvid'ov{\'{a}}, H.~Mizerov{\'{a}},
  and A.~{\'{S}}wierczewska-Gwiazda.
\newblock {Existence of global weak solutions to the kinetic Peterlin model}.
\newblock {\em Nonlinear Anal.-Real}, 44:465--478, 2018.

\bibitem{He2010}
L.~He and L.~Xu.
\newblock Global well-posedness for viscoelastic fluid system in bounded
  domains.
\newblock {\em {SIAM} J. Math. Anal.}, 42(6):2610--2625, 2010.

\bibitem{cms/1199377557}
D.~Hu and T.~Lelièvre.
\newblock {New entropy estimates for the Oldroyd-B model and related models}.
\newblock {\em Commun. Math. Sci.}, 5(4):909 -- 916, 2007.

\bibitem{Hulsen.1990}
M.~A. Hulsen.
\newblock {A sufficient condition for a positive definite configuration tensor
  in differential models}.
\newblock {\em {J. Non-Newton. Fluid}}, 38(1):93--100, 1990.

\bibitem{joseph}
D.~Joseph.
\newblock {\em Fluid Dynamics of Viscoelastic Liquids}.
\newblock Springer, New York, 1990.

\bibitem{Kalousek2019}
M.~Kalousek.
\newblock On dissipative solutions to a system arising in viscoelasticity.
\newblock {\em J Math. Fluid Mech.}, 21(4), 2019.

\bibitem{Kim.2006}
H.~Kim.
\newblock {A blow-up criterion for the nonhomogeneous incompressible
  Navier--Stokes equations}.
\newblock {\em {SIAM J. Math. Anal.}}, 37(5):1417--1434, 2006.

\bibitem{Larsonbook}
R.~Larson.
\newblock {\em Constitutive Equations for Polymer Melts and Solutions}.
\newblock Elsevier, 1988.

\bibitem{lei2007b}
Z.~Lei, C.~Liu, and Y.~Zhou.
\newblock Global existence for a 2{D} incompressible viscoelastic model with
  small strain.
\newblock {\em Commun. Math. Sci.}, 5:595--616, 2007.

\bibitem{lei2007a}
Z.~Lei, C.~Liu, and Y.~Zhou.
\newblock Global solutions for incompressible viscoelastic fluids.
\newblock {\em Arch. Ration. Mech. Anal.}, 188(3):371--398, 2007.

\bibitem{Lei2010}
Z.~Lei, N.~Masmoudi, and Y.~Zhou.
\newblock Remarks on the blowup criteria for {O}ldroyd models.
\newblock {\em J. Differ. Equ.}, 248(2):328--341, 2010.

\bibitem{Lin2005}
F.-H. Lin, C.~Liu, and P.~Zhang.
\newblock On hydrodynamics of viscoelastic fluids.
\newblock {\em Commun. Pure Appl. Math.}, 58(11):1437--1471, 2005.

\bibitem{LIONS2000}
P.~L. Lions and N.~Masmoudi.
\newblock Global solutions for some {O}ldroyd models of non-newtonian flows.
\newblock {\em Chin. Ann. Math.}, 21(02):131--146, 2000.

\bibitem{Lu.2018}
Y.~Lu and Z.~Zhang.
\newblock {Relative entropy, weak-strong uniqueness, and conditional regularity
  for a compressible Oldroyd--B model}.
\newblock {\em {SIAM J. Math. Anal.}}, 50(1):557--590, 2018.

\bibitem{LukacovaMedvidova.2015}
M.~Luk{\'a}{\v{c}}ov{\'a}-Medvid'ov{\'a}, H.~Mizerov{\'a}, and
  {\v{S}}.~Ne{\v{c}}asov{\'a}.
\newblock {Global existence and uniqueness result for the diffusive Peterlin
  viscoelastic model}.
\newblock {\em {Nonlinear Anal.-Theor.}}, 120:154--170, 2015.

\bibitem{LukacovaMedvidova.2017}
M.~Luk{\'a}{\v{c}}ov{\'a}-Medvid'ov{\'a}, H.~Mizerov{\'a},
  {\v{S}}.~Ne{\v{c}}asov{\'a}, and M.~Renardy.
\newblock {Global existence result for the generalized Peterlin viscoelastic
  model}.
\newblock {\em {SIAM J. Math. Anal.}}, 49(4):2950--2964, 2017.

\bibitem{LukacovaMedvidova.2017a}
M.~Luk{\'a}{\v{c}}ov{\'a}-Medvid'ov{\'a}, H.~Mizerov{\'a}, H.~Notsu, and
  M.~Tabata.
\newblock {Numerical analysis of the Oseen-type Peterlin viscoelastic model by
  the stabilized Lagrange--Galerkin method. Part II: A linear scheme}.
\newblock {\em {ESAIM Math. Model. Numer. Anal.}}, 51(5):1663--1689, 2017.

\bibitem{LukovMedvidov2017b}
M.~Luk{\'{a}}{\v{c}}ov{\'{a}}{\textendash}Medvid'ov{\'{a}}, H.~Mizerov{\'{a}},
  H.~Notsu, and M.~Tabata.
\newblock Numerical analysis of the {O}seen-type {P}eterlin viscoelastic model
  by the stabilized {L}agrange{\textendash}{G}alerkin method. {P}art {I}: {A}
  nonlinear scheme.
\newblock {\em {ESAIM Math. Model. Numer. Anal.}}, 51(5):1637--1661, 2017.

\bibitem{Malek.2018}
J.~M{\'a}lek, V.~Pr{\r{u}}{\v{s}}a, T.~Sk{\v{r}}ivan, and E.~S{\"u}li.
\newblock {Thermodynamics of viscoelastic rate-type fluids with stress
  diffusion}.
\newblock {\em {Phys. Fluids}}, 30(2):023101, 2018.

\bibitem{Masmoudi.2011}
N.~Masmoudi.
\newblock {Global existence of weak solutions to macroscopic models of
  polymeric flows}.
\newblock {\em {J. Math. Pure Appl.}}, 96(5):502--520, 2011.

\bibitem{Mizerova.2015}
H.~Mizerov{\'a}.
\newblock {\em {Analysis and numerical solution of the Peterlin viscoelastic
  model}}.
\newblock {Dissertation}, {Johannes Gutenberg-Universit{\"a}t}, Mainz, 2015.

\bibitem{Novotny.2004}
A.~Novotn\'y and I.~Stra\v{s}kraba.
\newblock {\em {Introduction to the mathematical theory of compressible flow}}.
\newblock {Oxford Lecture Series in Mathematics and its Applications, 27}.
  {Oxford University Press, Oxford}, 2004.

\bibitem{Peterlin}
A.~Peterlin.
\newblock Hydrodynamics of macromolecules in a velocity field with longitudinal
  gradient.
\newblock {\em J. Poly. Sci. Pol. Lett.}, 4(4):287--291, 1966.

\bibitem{Pruess.2016}
J.~Pruess and G.~Simonett.
\newblock {\em {Moving Interfaces and Quasilinear Parabolic Evolution
  Equations}}, volume 105.
\newblock Birkh\"auser, 2016.

\bibitem{renardy2000mathematical}
M.~Renardy.
\newblock {\em Mathematical analysis of viscoelastic flows}.
\newblock SIAM, 2000.

\bibitem{Serrin.1962}
J.~Serrin.
\newblock {On the interior regularity of weak solutions of the Navier-Stokes
  equations}.
\newblock {\em {Arch. Ration. Mech. An.}}, 9:187--195, 1962.

\bibitem{Siginer2014}
D.~A. Siginer.
\newblock {\em Stability of Non-Linear Constitutive Formulations for
  Viscoelastic Fluids}.
\newblock Springer, 2014.

\bibitem{Struwe.1988}
M.~Struwe.
\newblock {On partial regularity results for the Navier--Stokes equations}.
\newblock {\em {Commun. Pure Appl. Math.}}, 41:437--458, 1988.

\bibitem{Zhang2012}
T.~Zhang and D.~Fang.
\newblock Global existence of strong solution for equations related to the
  incompressible viscoelastic fluids in the critical ${L}^p$ framework.
\newblock {\em {SIAM} J. Math. Anal.}, 44(4):2266--2288, 2012.

\end{thebibliography}
	\bibliographystyle{abbrv}
\end{document}